\documentclass[a4paper,fleqn,11pt]{amsart}

\usepackage{mathrsfs}
\usepackage{cite}
\usepackage{amssymb}
\usepackage{graphicx}
\usepackage{xcolor}

\DeclareMathAlphabet{\mathpzc}{OT1}{pzc}{m}{it}

\newtheorem{theorem}{Theorem}[section]
\newtheorem{lemma}[theorem]{Lemma}
\newtheorem{claim}[theorem]{Claim}

\newtheorem{proposition}[theorem]{Proposition}
\theoremstyle{definition}
\newtheorem{definition}[theorem]{Definition}
\newtheorem{example}[theorem]{Example}
\theoremstyle{remark}
\newtheorem{remark}[theorem]{Remark}
\newtheorem{notation}[theorem]{Notation}

\numberwithin{equation}{section}

\allowdisplaybreaks[4]

\begin{document}

\title{Whitney regularity and Thom condition for families of non-isolated mixed singularities} 

\author{Christophe Eyral and Mutsuo Oka}
\address{C. Eyral, Institute of Mathematics, Polish Academy of Sciences, \'Sniadeckich~8, 00-656 Warsaw, Poland}
\email{eyralchr@yahoo.com}
\address{M. Oka, Department of Mathematics, Tokyo University of Science, 1-3 Kagurazaka, Shinjuku-ku, Tokyo 162-8601, Japan}   
\email{oka@rs.kagu.tus.ac.jp}



\begin{abstract}
We investigate the equisingularity question for $1$-parameter deformation families of mixed polynomial functions $f_t(\mathbf{z},\bar{\mathbf{z}})$ from the Newton polygon point of view. We show that if the members $f_t$ of the family satisfy a number of elementary conditions, which can be easily described in terms of the Newton polygon, then the corresponding family of mixed hypersurfaces $f_t^{-1}(0)$ is Whitney equisingular (and hence topologically equisingular) and satisfies the Thom condition.
\end{abstract}

\maketitle

\markboth{C. Eyral and M. Oka}{Whitney regularity and Thom condition for families of non-isolated mixed singularities} 

\section{Introduction}

Consider a $1$-parameter deformation family $\{f_t\}$ of mixed polynomial functions $f_t(\mathbf{z},\bar{\mathbf{z}})$, where $\mathbf{z}:=(z_1,\ldots,z_n)\in\mathbb{C}^n$, $\bar{\mathbf{z}}:=(\bar{z}_1,\ldots,\bar{z}_n)\in\mathbb{C}^n$ and $t\in\mathbb{C}$. (Here, $\bar{z}_i$ denotes the complex conjugate of $z_i$.) A central question in  equisingularity theory is to find easy-to-check conditions on the members $f_t$ of the family $\{f_t\}$ that guarantee that the corresponding family  of mixed hypersurfaces 
\begin{equation*}
V(f_t):=f_t^{-1}(0)\subseteq \mathbb{C}^n
\end{equation*}
is ``equisingular''---i.e., for all small $t$, the mixed singularities of $V(f_t)$ are ``equivalent'' (in a sense that should be duly specified) to those of $V(f_0)$ in a neighbourhood~of the origin $\mathbf{0}\in\mathbb{C}^n$. In the present paper, we investigate this question from the Newton polygon point of view. For example, in the special case of \emph{convenient} functions (i.e., functions whose Newton polygon intersects each coordinates axis), we prove that if for all sufficiently small $t$, the Newton polygon of $f_t$ is independent of $t$ and $f_t$ is strongly non-degenerate (i.e., the face functions defined by $f_t$ and the compact faces of the Newton polygon have no critical points on the torus $(\mathbb{C}^*)^n$), then the family of mixed hypersurfaces $\{V(f_t)\}$ is Whitney equisingular. That is, there exists a \emph{real} Whitney $(b)$-regular stratification of the mixed hypersurface 
\begin{equation*}
V(f):=f^{-1}(0)\subseteq \mathbb{C}\times\mathbb{C}^n
\end{equation*}
such that the $t$-axis is a stratum, where $f(t,\mathbf{z},\bar{\mathbf{z}})$ is the underlying mixed polynomial function defining the family $\{f_t\}$ (i.e., $f_t(\mathbf{z},\bar{\mathbf{z}})=f(t,\mathbf{z},\bar{\mathbf{z}})$). Here, $\mathbb{C}\times\mathbb{C}^n$ is identified with $\mathbb{R}^2\times\mathbb{R}^{2n}$, and $V(f)$ is understood as a real algebraic variety in $\mathbb{R}^2\times\mathbb{R}^{2n}$. Note that Whitney equisingularity is quite a strong form of equisingularity. For example, combined with the Thom-Mather first isotopy theorem, it implies that the family $\{V(f_t)\}$ is topologically equisingular (i.e., the local ambient topological type of $V(f_t)$ at $\mathbf{0}\in\mathbb{C}^n$  is constant as $t$ varies). 

Another regularity condition which is important to know how to detect is the Thom $a_f$ condition. We shall see that families of convenient strongly non-degenerate mixed polynomial functions with constant Newton polygon also satisfy this condi\penalty 10000 tion.

Strongly non-degenerate mixed polynomial functions which satisfy the ``convenience'' assumption have, at worst, an \emph{isolated} mixed singularity at the origin, and for this reason they behave quite well. For non-convenient functions and non-isolated mixed singularities, the equisingularity question is more subtle. 
Indeed, \emph{a priori}, for such functions, the strong non-degeneracy (which only controls the behaviour of the face functions corresponding to the \emph{compact} faces of the Newton polygon) seems to be insufficient to get the Whitney $(b)$-regularity and the Thom $a_f$ condition. In this case, we need an additional condition in order to control the behaviour of the face functions corresponding to the ``essential'' non-compact faces. Roughly this condition is as follows. Any non-compact face $\Delta$ has a ``non-compact direction'' which is characterized by a subset of indices $I_\Delta=\{i_1,\ldots,i_m\}\subseteq\{1,\ldots, n\}$. We say that $\Delta$ is \emph{essential} if the restriction of $f_t$ to the coordinates subspace
\begin{equation*}
\{(z_1,\ldots,z_n)\in\mathbb{C}^n\mid z_{i}=0 
\mbox{ if } i\notin I_\Delta\}
\end{equation*}
identically vanishes. The additional condition we require is as follows: for each $t$ and each essential non-compact face $\Delta$, we need that the face function $(f_t)_\Delta$ defined by $f_t$ and $\Delta$ is \emph{locally tame}---that is, for any fixed non-zero complex numbers $u_{i_1},\ldots, u_{i_m}$ close enough to the origin, $(f_t)_\Delta$ has  no critical point on the set 
\begin{equation*}
\{(z_1,\ldots,z_n)\in(\mathbb{C}^*)^n\mid z_{i_j} 
= u_{i_j} \mbox{ for } 1\leq j\leq m\}
\end{equation*}
as a function of the $n-m$ variables $z_{i_{m+1}},\ldots, z_{i_n}$. (Here, $i_{m+1},\ldots,i_n$ denote the elements of the set $\{1,\ldots, n\}$ which are not in $I_\Delta=\{i_1,\ldots,i_m\}$.) Moreover, we assume that this condition is satisfied \emph{uniformly} with respect to the parameter $t$---so-called \emph{uniform local tameness}. 
Our main result says that if for all sufficiently small $t$, the Newton polygon of $f_t$ is independent of $t$, if $f_t$ is strongly non-degenerate, and if the family of functions $\{f_t\}$ is uniformly locally tame, then the corresponding family of mixed hypersurfaces $\{V(f_t)\}$ is Whitney equisingular and satisfies the Thom  $a_f$ condition---even if the functions $f_t$ are not convenient and have non-isolated mixed singularities.
Note that the ``holomorphic'' counterpart of this result was proved by the authors in \cite{EO}. However, as mixed hypersurfaces do not carry any complex structure, the proof for the ``mixed'' case requires essential new arguments. 

The paper is organized as follows. In Section \ref{sect-ncnblt}, we recall important definitions and tools about mixed polynomial functions and mixed singularities. In Section~\ref{sect-ult}, we prove two important preliminary propositions (Propositions \ref{lemmasmooth} and \ref{lemma-transversality}) and we state our main result (Theorem \ref{mt2}). Finally, in Section \ref{pmt2}, we give the proof of Theorem \ref{mt2}.

\section{Strong non-degeneracy and local tameness}\label{sect-ncnblt}

In this section, we recall important definitions and tools about mixed polynomial functions which have been introduced by the second author in \cite{O7,O3,O1} and that will be used in the present paper.

\subsection{Mixed singularities}

Let $\mathbb{C}^n$ denotes the complex space generated by $n$ complex variables $(z_1,\ldots, z_n)$, and let $f(\mathbf{z},\bar{\mathbf{z}})$ be a complex-valued polynomial function of the variables $\mathbf{z}:=(z_1,\ldots, z_n)$ and $\bar{\mathbf{z}}:=(\bar{z}_1,\ldots,\bar{z}_n)$.   Write
\begin{equation*}
f(\mathbf{z},\bar{\mathbf{z}})=\sum_{\nu,\mu} c_{\nu,\mu}\, 
\mathbf{z}^\nu \bar{\mathbf{z}}^\mu,
\end{equation*}
where $\nu:=(\nu_1,\ldots,\nu_n)$ and $\mu:=(\mu_1,\ldots,\mu_n)$ are $n$-tuples of non-negative integers, $c_{\nu,\mu}$ is a complex coefficient (depending on $\nu$ and $\mu$), and $\mathbf{z}^\nu$ (respectively, $\bar{\mathbf{z}}^\mu$) is the monomial $z_1^{\nu_1}\cdots z_n^{\nu_n}$ (respectively, $\bar{z}_1^{\mu_1}\cdots \bar{z}_n^{\mu_n}$). 
Such a polynomial function $f(\mathbf{z},\bar{\mathbf{z}})$ is called a \emph{mixed} polynomial function of the variables $\mathbf{z}=(z_1,\ldots, z_n)$. The set
\begin{equation*}
V(f):=\{\mathbf{z}\in \mathbb{C}^n \mid f(\mathbf{z},\bar{\mathbf{z}})=0\}
\end{equation*}
is called the \emph{mixed ``hypersurface''} defined by $f$. Hereafter, we always assume that $f$ does not identically vanish near the origin $\mathbf{0}\in\mathbb{C}^n$ and we suppose that the coefficient $c_{\mathbf{0},\mathbf{0}}$ is zero so that $\mathbf{0}\in V(f)$. 

For every $1\leq i\leq n$, let $x_i:=\Re(z_i)$ and $y_i:=\Im(z_i)$ be the real and imaginary parts of $z_i$, respectively, and let $\mathbf{x}:=(x_1,\ldots,x_n)\in\mathbb{R}^n$ and  $\mathbf{y}:=(y_1,\ldots,y_n)\in\mathbb{R}^n$, so that $\mathbf{z}=\mathbf{x}+\sqrt{-1}\, \mathbf{y}$. Hereafter, the spaces $\mathbb{C}^n$ and $\mathbb{R}^{2n}$ are identified through the 1-1 correspondence $\mathbf{z}\in\mathbb{C}^n\mapsto \mathbf{z}_{\mathbb{R}}:=(\mathbf{x},\mathbf{y})\in \mathbb{R}^{2n}$. Under this identification, the Hermitian inner product in $\mathbb{C}^n$ (denoted by $\langle\cdot,\cdot\rangle$) and the Euclidean inner product in $\mathbb{R}^{2n}$ (denoted by $\langle\cdot,\cdot\rangle_{\mathbb{R}}$) are related by $\langle\mathbf{z}_{\mathbb{R}},\mathbf{z}'_{\mathbb{R}}\rangle_{\mathbb{R}}=\Re\langle\mathbf{z},\mathbf{z}'\rangle$.

If we divide $f(\mathbf{z},\bar{\mathbf{z}})$ into its real and imaginary parts, denoted by
\begin{equation*}
g(\mathbf{x},\mathbf{y}):=\Re (f(\mathbf{z},\bar{\mathbf{z}}))
\quad\mbox{and}\quad
h(\mathbf{x},\mathbf{y}):=\Im (f(\mathbf{z},\bar{\mathbf{z}})),
\end{equation*}
respectively, then we may view $f$ as the polynomial function $g(\mathbf{x},\mathbf{y})+\sqrt{-1}\, h(\mathbf{x},\mathbf{y})$ of $2n$ real variables $(\mathbf{x},\mathbf{y})=(x_1,\ldots,x_n,y_1,\ldots,y_n)$, and $V(f)$ may be understood as the real algebraic variety in $\mathbb{R}^{2n}$ defined by the equations $g(\mathbf{x},\mathbf{y})=h(\mathbf{x},\mathbf{y})=0$.  By abuse of notation, we may also write $f(\mathbf{x},\mathbf{y})=g(\mathbf{x},\mathbf{y})+\sqrt{-1}\, h(\mathbf{x},\mathbf{y})$. 
 Note that the real and imaginary parts of $f$ may also be viewed as the mixed polynomial functions 
\begin{equation*}
\mathbf{z}\mapsto g\biggl(\frac{\mathbf{z}+\bar{\mathbf{z}}}{2},\frac{\mathbf{z}-\bar{\mathbf{z}}}{2\sqrt{-1}}\biggr)
\quad\mbox{and}\quad
\mathbf{z}\mapsto h\biggl(\frac{\mathbf{z}+\bar{\mathbf{z}}}{2},\frac{\mathbf{z}-\bar{\mathbf{z}}}{2\sqrt{-1}}\biggr)
\end{equation*} 
respectively.
By abuse of notation, we still write $g(\mathbf{z},\bar{\mathbf{z}})$ for the mixed function $g((\mathbf{z}+\bar{\mathbf{z}})/2,(\mathbf{z}-\bar{\mathbf{z}})/2\sqrt{-1})=g(\mathbf{x},\mathbf{y})$. Similarly for $h$.

We say that a point $\mathbf{z}_0=\mathbf{x}_0+\sqrt{-1}\, \mathbf{y}_0\in\mathbb{C}^n$ is a \emph{critical point} of the mixed polynomial function $f\colon \mathbb{C}^n\to\mathbb{C}$, $\mathbf{z}\mapsto f(\mathbf{z},\bar{\mathbf{z}})$, if the vectors
\begin{align*}
dg(\mathbf{x}_0,\mathbf{y}_0):=\biggl(\frac{\partial g}{\partial x_1}(\mathbf{x}_0,\mathbf{y}_0),\ldots, \frac{\partial g}{\partial x_n}(\mathbf{x}_0,\mathbf{y}_0), \frac{\partial g}{\partial y_1}(\mathbf{x}_0,\mathbf{y}_0),\ldots, \frac{\partial g}{\partial y_n}(\mathbf{x}_0,\mathbf{y}_0)\biggr)
\end{align*}
and
\begin{align*}
dh(\mathbf{x}_0,\mathbf{y}_0):=\biggl(\frac{\partial h}{\partial x_1}(\mathbf{x}_0,\mathbf{y}_0),\ldots, \frac{\partial h}{\partial x_n}(\mathbf{x}_0,\mathbf{y}_0), \frac{\partial h}{\partial y_1}(\mathbf{x}_0,\mathbf{y}_0),\ldots, \frac{\partial h}{\partial y_n}(\mathbf{x}_0,\mathbf{y}_0)\biggr)
\end{align*}
of $\mathbb{R}^{2n}$ are linearly dependent over $\mathbb{R}$. 
By \cite[Proposition 1]{O7} (see also \cite{O3,O1}),
the above condition is equivalent to anyone of the following two conditions:
\begin{enumerate}
\item
the vectors $\bar {\partial} g (\mathbf{z}_0,\bar{\mathbf{z}}_0)$ and $\bar {\partial} h (\mathbf{z}_0,\bar{\mathbf{z}}_0)$ of $\mathbb{C}^{n}$ are linearly dependent over $\mathbb{R}$;
\item
there exists a complex number $\lambda$, with $\vert\lambda\vert=1$, such that 
\begin{equation*}
\qquad \overline{\partial f(\mathbf{z}_0,\bar{\mathbf{z}}_0)} =
\lambda \bar {\partial} f(\mathbf{z}_0,\bar{\mathbf{z}}_0).
\end{equation*}
\end{enumerate}
Here, if $k(\mathbf{z},\bar{\mathbf{z}})$ is any mixed polynomial function, then the vectors $\partial k(\mathbf{z},\bar{\mathbf{z}})$ and $\bar{\partial} k(\mathbf{z},\bar{\mathbf{z}})$ are defined by
\begin{equation*}
\left\{
\begin{aligned}
& \partial k(\mathbf{z},\bar{\mathbf{z}}):=\biggl(\frac{\partial k}{\partial z_1}(\mathbf{z},\bar{\mathbf{z}}),\ldots, \frac{\partial k}{\partial z_n}(\mathbf{z},\bar{\mathbf{z}})\biggr),
\\
& \bar{\partial} k(\mathbf{z},\bar{\mathbf{z}}):=\biggl(\frac{\partial k}{\partial \bar{z}_1}(\mathbf{z},\bar{\mathbf{z}}),\ldots, \frac{\partial k}{\partial \bar{z}_n}(\mathbf{z},\bar{\mathbf{z}})\biggr).
\end{aligned}
\right.
\end{equation*}
Hereafter, we shall often use the simplified notation $\overline{\partial k}(\mathbf{z},\bar{\mathbf{z}}):=\overline{\partial k(\mathbf{z},\bar{\mathbf{z}})}$.  One must not confuse $\overline{\partial k}(\mathbf{z},\bar{\mathbf{z}})$ with $\bar{\partial} k(\mathbf{z},\bar{\mathbf{z}})$. Note that in the special case where $k$ is a \emph{real-valued} mixed polynomial function, we have $\overline{\partial k}(\mathbf{z},\bar{\mathbf{z}})=\bar{\partial}k(\mathbf{z},\bar{\mathbf{z}})$.
Coming back to $f=g+\sqrt{-1}h$, we recall that 
\begin{equation*}
\left\{
\begin{aligned}
& \frac{\partial f}{\partial z_i} (\mathbf{z},\bar{\mathbf{z}}) = 
\frac{\partial g}{\partial z_i}(\mathbf{z},\bar{\mathbf{z}}) +\sqrt{-1} \frac{\partial h}{\partial z_i} (\mathbf{z},\bar{\mathbf{z}}),
\\
& \frac{\partial f}{\partial \bar{z}_i}(\mathbf{z},\bar{\mathbf{z}}) = 
\frac{\partial g}{\partial \bar{z}_i}(\mathbf{z},\bar{\mathbf{z}}) +\sqrt{-1} \frac{\partial h}{\partial \bar{z}_i}(\mathbf{z},\bar{\mathbf{z}}),
\end{aligned}
\right.
\end{equation*}
where
\begin{equation*}
\left\{
\begin{aligned}
& \frac{\partial g}{\partial z_i} (\mathbf{z},\bar{\mathbf{z}}) = \frac{1}{2} \Bigl( \frac{\partial g}{\partial x_i}(\mathbf{x},\mathbf{y}) -\sqrt{-1} \frac{\partial g}{\partial y_i}(\mathbf{x},\mathbf{y}) \Bigr), \\
& \frac{\partial g}{\partial \bar{z}_i} (\mathbf{z},\bar{\mathbf{z}}) = \frac{1}{2} \Bigl( \frac{\partial g}{\partial x_i}(\mathbf{x},\mathbf{y}) +\sqrt{-1} \frac{\partial g}{\partial y_i}(\mathbf{x},\mathbf{y}) \Bigr),
\end{aligned}
\right.
\end{equation*}
and similarly for $h$.

If $\mathbf{z}_0$ is a critical point of $f$, and if furthermore $\mathbf{z}_0\in V(f)$, then we say that $\mathbf{z}_0$ is a \emph{mixed singular point} or a \emph{mixed singularity} of the mixed hypersurface $V(f)$. 
We say that $V(f)$ is \emph{mixed non-singular} if it has no mixed singular point. If $V(f)$ is mixed non-singular, then, necessarily, it is a smooth algebraic variety of real codimension two (cf.~\cite[\S 3.1]{O3}). Note that a singular point of $V(f)$ as a point of the real algebraic variety $V(f)$---that is, a singular point of $V(f)$ understood as the variety $\{(\mathbf{x},\mathbf{y})\in\mathbb{R}^{2n}\, ;\, g(\mathbf{x},\mathbf{y})=h(\mathbf{x},\mathbf{y})=0\}$---is always a mixed singular point of the mixed hypersurface $V(f)$. However the converse is not true.
For example, every point of the sphere  $z_1\bar{z}_1+\cdots + z_n\bar{z}_n=1$ is a mixed singular point. 

The tangent space $T_{\mathbf{a}}V(g)$ at a mixed non-singular point $\mathbf{a}\in \mathbb{C}^{n}$ of the mixed hypersurface $V(g)$ defined by the real-valued mixed function $g(\mathbf{z},\bar{\mathbf{z}})$ is the real subspace of $\mathbb{C}^{n}$ given by
\begin{align*}
T_{\mathbf{a}}V(g) & =\{\mathbf{z}_{\mathbb{R}}\in\mathbb{R}^{2n}\mid \langle \mathbf{z}_{\mathbb{R}}, dg(\mathbf{a}_{\mathbb{R}}) \rangle_{\mathbb{R}}=0\} \\
& = \{\mathbf{z}\in\mathbb{C}^{n}\mid \Re\langle \mathbf{z}, \overline{\partial g}(\mathbf{a},\bar{\mathbf{a}}) \rangle=\Re\langle \mathbf{z}, \bar{\partial} g(\mathbf{a},\bar{\mathbf{a}}) \rangle=0\}.
\end{align*}
In other words, $T_{\mathbf{a}}V(g)$ is the real subspace of $\mathbb{C}^n$ whose vectors are orthogonal in $\mathbb{R}^{2n}$ to the vector $\bar{\partial}g(\mathbf{a},\bar{\mathbf{a}})$.
Similarly for $T_{\mathbf{a}}V(h)$. The tangent space $T_{\mathbf{a}}V(f)$ at a mixed non-singular point $\mathbf{a}\in \mathbb{C}^{n}$ of the mixed hypersurface $V(f)$ defined by the complex-valued mixed function $f(\mathbf{z},\bar{\mathbf{z}})$ is the real subspace of $\mathbb{C}^{n}$ whose vectors are orthogonal in $\mathbb{R}^{2n}$ to the vectors $\bar{\partial}g(\mathbf{a},\bar{\mathbf{a}})$ and $\bar{\partial}h(\mathbf{a},\bar{\mathbf{a}})$. That is,
\begin{align*}
T_{\mathbf{a}}V(f)=
\bar{\partial}g(\mathbf{a},\bar{\mathbf{a}})^{\bot}
\cap \bar{\partial}h(\mathbf{a},\bar{\mathbf{a}})^{\bot},
\end{align*}
where $\bar{\partial}g(\mathbf{a},\bar{\mathbf{a}})^{\bot}$ (respectively, 
$\bar{\partial}h(\mathbf{a},\bar{\mathbf{a}})^{\bot}$) denotes the real orthogonal complement of the real subspace of $\mathbb{C}^n$ generated by 
$\bar{\partial}g(\mathbf{a},\bar{\mathbf{a}})$ 
(respectively, by $\bar{\partial}h(\mathbf{a},\bar{\mathbf{a}})$).

\subsection{Non-compact Newton boundary}
The \emph{Newton polygon} of $f(\mathbf{z},\bar{\mathbf{z}})$ at the origin (denoted by $\Gamma_{\! +}(f;\mathbf{z},\bar{\mathbf{z}})$) is the convex hull in $\mathbb{R}_+^n$ of the set
\begin{equation*}
\bigcup_{c_{\nu,\mu}\not=0} ((\nu+\mu)+\mathbb{R}_+^n),
\end{equation*}
where $\mathbb{R}_+^n=\{\xi:=(\xi_1,\ldots, \xi_n)\in \mathbb{R}^n\mid \xi_i\geq 0 \hbox{ for } 1\leq i\leq n\}$. 
The \emph{Newton boundary} of $f(\mathbf{z},\bar{\mathbf{z}})$ at the origin (denoted by $\Gamma(f;\mathbf{z},\bar{\mathbf{z}})$) is the union of the \emph{compact} faces of $\Gamma_{\! +}(f;\mathbf{z},\bar{\mathbf{z}})$. 
Note that if $f$ is a complex analytic function, then these definitions coincide with the standard ones.

For any system of weights $\mathbf{w}:=(w_1,\ldots,w_n)\in\mathbb{N}^n\setminus\{\mathbf{0}\}$, there is a linear map $\mathbb{R}^n\to\mathbb{R}$ given by
\begin{equation*}
\xi:=(\xi_1,\ldots, \xi_n)\mapsto\sum_{1\leq i\leq n} \xi_i w_i.
\end{equation*}
Let $l_\mathbf{w}$ be the restriction of this map to $\Gamma_{+}(f;\mathbf{z},\bar{\mathbf{z}})$, let $d_\mathbf{w}$ be the minimal value of $l_\mathbf{w}$, and let 
 $\Delta_\mathbf{w}$ be the (possibly non-compact) face of $\Gamma_{+}(f;\mathbf{z},\bar{\mathbf{z}})$ defined by the locus where $l_\mathbf{w}$ takes this minimal value. It is easy to see that
\begin{equation*}
d_\mathbf{w}=\mbox{min} \{ \mbox{rdeg}_{\mathbf{w}}
(\mathbf{z}^\nu \bar{\mathbf{z}}^\mu) \mid c_{\nu,\mu}\not=0\},
\end{equation*}
where $\mbox{rdeg}_{\mathbf{w}}(\mathbf{z}^\nu \bar{\mathbf{z}}^\mu)$ is the \emph{radial degree} of the monomial $\mathbf{z}^\nu \bar{\mathbf{z}}^\mu$ with respect to the weights $\mathbf{w}$, which is defined by
\begin{equation*}
\mbox{rdeg}_{\mathbf{w}}(\mathbf{z}^\nu \bar{\mathbf{z}}^\mu):=\sum_{1\leq i\leq n} w_i(\nu_i+\mu_i)=l_\mathbf{w}(\nu+\mu).
\end{equation*}
Note that $\Delta_\mathbf{w}=\{\xi\in \Gamma_{+}(f;\mathbf{z},\bar{\mathbf{z}})\mid l_\mathbf{w}(\xi)=d_\mathbf{w}\}$, and if $w_i>0$ for each $1\leq i\leq n$, then $\Delta_\mathbf{w}$ is a (compact) face of the Newton boundary $\Gamma(f;\mathbf{z},\bar{\mathbf{z}})$.

\begin{notation}
For any subset $I\subseteq\{1,\ldots, n\}$, we set
\begin{align*}
& \mathbb{C}^I:=\{(z_1,\ldots, z_n)\in \mathbb{C}^n\mid z_i=0 \mbox{ if } i\notin I\}, \\
& \mathbb{C}^{*I}:=\{(z_1,\ldots, z_n)\in \mathbb{C}^n\mid z_i=0 \mbox{ if and only if } i\notin I\}.
\end{align*}
In particular, $\mathbb{C}^{\emptyset}=\mathbb{C}^{*\emptyset}=\{\mathbf{0}\}$ and $\mathbb{C}^{*\{1,\ldots,n\}}=(\mathbb{C}^*)^n$. (As usual, $\mathbb{C}^*:=\mathbb{C}\setminus \{\mathbf{0}\}$.)
\end{notation}

\begin{definition}\label{def-mnbencf}
The \emph{non-compact Newton boundary} of~$f(\mathbf{z},\bar{\mathbf{z}})$ at the origin (denoted by $\Gamma_{nc}(f;\mathbf{z},\bar{\mathbf{z}})$) is obtained from the usual Newton boundary $\Gamma(f;\mathbf{z},\bar{\mathbf{z}})$ by adding the ``essential'' non-compact faces of $\Gamma_+(f;\mathbf{z},\bar{\mathbf{z}})$:
\begin{align*}
\Gamma_{nc}(f;\mathbf{z},\bar{\mathbf{z}}):=\Gamma(f;\mathbf{z},\bar{\mathbf{z}})
\cup \{\mbox{``essential'' non-compact faces}\}.
\end{align*}
Here, a non-compact face $\Delta$ of $\Gamma_+(f;\mathbf{z},\bar{\mathbf{z}})$ is said to be \emph{essential} if there are weights $\mathbf{w}:=(w_1,\ldots,w_n)\in\mathbb{N}^n\setminus \{\mathbf{0}\}$ such that the following three conditions hold:
\begin{enumerate}
\item[(i)]
$\Delta=\Delta_\mathbf{w}$ (i.e., $\Delta$ is the face defined by the locus where $l_\mathbf{w}$ takes its minimal value);
\item[(ii)]
$f_{\mid \mathbb{C}^{I_\mathbf{w}}}\equiv 0$, where $I_\mathbf{w}:=\bigl\{i\in\{1,\ldots,n\}\ ;\, w_i=0\bigr\}$;
\item[(iii)]
for any $i\in I_\mathbf{w}$ and any point $\alpha\in\Delta$, the half-line $\alpha+\mathbb{R}_{+}\mathbf{e}_i$ is contained in $\Delta$, where $\mathbf{e}_i$ is the unit vector in the direction of the $\xi_i$-axis.
\end{enumerate}
\end{definition}

The set $I_\mathbf{w}$ does not depend on the choice of the weights $\mathbf{w}$. It is called the \emph{non-compact direction} of $\Delta$ and is denoted by $I_\Delta$.

\begin{example}\label{example1}
If $f(z_1,z_2,\bar{z}_1,\bar{z}_2)=z_1^2z_2^2\bar{z}_2+z_1\bar{z}_1^2z_2^2+z_1^6$, then the non-compact face $\Delta:=(2,3)+\mathbb{R}_+\mathbf{e}_2$ is essential.
Indeed, we can take $\mathbf{w}=(1,0)$. Thus $\Delta=\Delta_{\mathbf{w}}$, $I_\Delta=\{2\}$ and $f(0,0,z_2,\bar{z}_2)=0$ for any $z_2$. On the other hand, the non-compact face $\Xi:=(6,0)+\mathbb{R}_+\mathbf{e}_1$ is not essential. Indeed, the non-compact direction of $\Xi$ is $I_\Xi=\{1\}$, and the function $(z_1,0)\mapsto f(z_1,\bar{z}_1,0,0)$ does not identically vanishes. Therefore, $\Gamma_{nc}(f;\mathbf{z},\bar{\mathbf{z}})$ has three (compact) $0$-dimensional faces (the points $A=(2,3)$, $B=(3,2)$ and $C=(6,0)$), two compact $1$-dimensional faces (the segment $\overline{AB}$ and $\overline{BC}$), and one $1$-dimensional essential non-compact face (the vertical half-line $\Delta$). See Figure~\ref{fig1}, left-hand side.
\end{example}

\begin{figure}[t]
\includegraphics[scale=1.8]{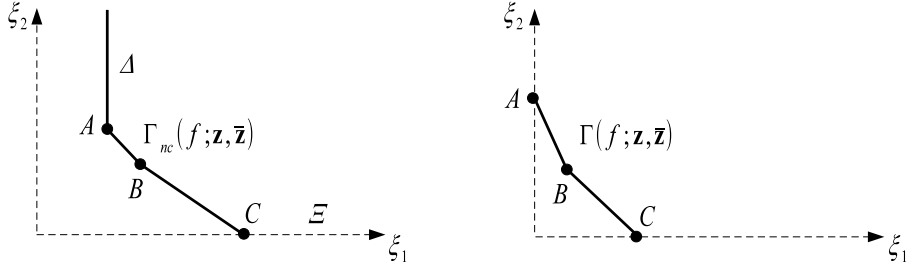}
\caption{Example \ref{example1} (left) and Example \ref{example2} (right)}
\label{fig1}
\end{figure}

\subsection{Strong non-degeneracy}
This notion plays a crucial role in the theory. 

\begin{definition}\label{def-nd}
The function $f$ is said to be \emph{strongly non-degenerate} if for any (compact) face $\Delta\subseteq \Gamma(f;\mathbf{z},\bar{\mathbf{z}})$, the face function
\begin{equation*}
f_\Delta(\mathbf{z},\bar{\mathbf{z}}):=\sum_{\nu+\mu\in \Delta} c_{\nu,\mu}\, \mathbf{z}^\nu \bar{\mathbf{z}}^\mu
\end{equation*}
(which is also a mixed polynomial function) has no critical 
point on $(\mathbb{C}^*)^n$.
\end{definition}

\begin{example}\label{example2}
If $f(z_1,z_2,\bar{z}_1,\bar{z}_2)=\bar{z}_2^4+z_1z_2^2+z_1^2\bar{z}_1$, then $\Gamma(f;\mathbf{z},\bar{\mathbf{z}})$ has three $0$-dimensional faces $A=(0,4)$ , $B=(1,2)$ and $C=(3,0)$, and two $1$-dimensional faces $\overline{AB}$ and $\overline{BC}$. See Figure \ref{fig1}, right-hand side. We check easily that the face functions $f_{A}=\bar{z}_2^4$, $f_{B}=z_1z_2^2$, $f_{C}=z_1^2\bar{z}_1$, $f_{\overline{AB}}=\bar{z}_2^4+z_1z_2^2$ and $f_{\overline{BC}}=z_1z_2^2+z_1^2\bar{z}_1$ have no critical point on $(\mathbb{C}^*)^2$, and therefore the function $f$ is strongly non-degenerate.
\end{example}

\begin{remark}\label{rk-ffpri}
Note that if $f=g+\sqrt{-1}\, h$ and $f_\Delta=g_\Delta+\sqrt{-1}\, h_\Delta$ are the decompositions of $f$ and $f_\Delta$ into their real and imaginary parts, then $g_\Delta$ (respectively, $h_\Delta$) is the face function defined by $g$ (respectively, by $h$) and $\Delta$.
\end{remark}

Let $\mathscr{I}_{nv}(f)$ (respectively, $\mathscr{I}_{v}(f)$) be the set of all subsets $I\subseteq \{1,\ldots,n\}$ such that $f_{\mid\mathbb{C}^I}\not\equiv 0$ 
(respectively, $f_{\mid\mathbb{C}^I}\equiv 0$). For any $I\in\mathscr{I}_{nv}(f)$ 
(respectively, any $I\in\mathscr{I}_{v}(f)$), the subspace $\mathbb{C}^I$ is called a \emph{non-vanishing} (respectively, a \emph{vanishing}) coordinates subspace for $f$.
The function $f$ is called \emph{convenient} if the singleton set $\{i\}$ belongs to $\mathscr{I}_{nv}(f)$ for any $1\leq i\leq n$. In other words, $f$ is convenient if and only if $\Gamma(f;\mathbf{z},\bar{\mathbf{z}})$ intersects with the $\xi_i$-axis of $\mathbb{R}^n$ for any $1\leq i\leq n$. 
If $f$ is a strongly non-degenerate mixed polynomial function, and if furthermore $f$ is convenient, then, in a small neighbourhood of the origin, the mixed hypersurface $V(f)$ is either mixed non-singular or has an isolated mixed singularity at $\mathbf{0}\in\mathbb{C}^n$. On the other hand, for non-convenient mixed polynomial functions, mixed singularities are not necessarily isolated. See \cite[\S 3]{O3}.

\subsection{Local tameness}
As we have seen in the previous section, the strong non-degeneracy controls the behaviour of the face functions corresponding to the compact faces of the Newton polygon. The role of the local tameness is to control the face functions corresponding to the essential non-compact faces.

\begin{notation}
For any $u_{i_1},\ldots, u_{i_m}\in\mathbb{C}^*$ ($m\leq n$), we set
\begin{equation*} 
\mathbb{C}^{*\{1,\ldots,n\}}_{u_{i_1},\ldots, u_{i_m}}:=
\bigl\{(z_1,\ldots,z_n)\in(\mathbb{C}^*)^n\mid z_{i_j}=u_{i_j} \mbox{ for } 1\leq j\leq m\bigr\}.
\end{equation*}
\end{notation}

\begin{definition}\label{definitionlocallytame}
Let $\Delta\subseteq\Gamma_{nc}(f;\mathbf{z},\bar{\mathbf{z}})$ be an essential non-compact face, and let $\mathbf{w}=(w_1,\ldots,w_n)$ be a system of weights satisfying the conditions (i), (ii) and (iii) of Definition \ref{def-mnbencf}. Suppose that $I_\Delta:=I_{\mathbf{w}}=\{i_1,\ldots,i_m\}$ (i.e., $w_i=0$ if and only if $i\in\{i_1,\ldots,i_m\}$). We say that the face function 
\begin{equation*}
f_\Delta(\mathbf{z},\bar{\mathbf{z}}):=
\sum_{\nu+\mu\in \Delta} c_{\nu,\mu}\, \mathbf{z}^\nu \bar{\mathbf{z}}^\mu
\end{equation*}
is \emph{locally tame} if there exists a positive number $r(f_\Delta)>0$ such that for any non-zero complex numbers $u_{i_1},\ldots, u_{i_m}\in\mathbb{C}^*$ with 
\begin{equation}\label{irlt}
\vert u_{i_1}\vert^2+\cdots +\vert u_{i_m}\vert^2 < r(f_\Delta)^2,
\end{equation} 
$f_\Delta$ has no critical point in $\mathbb{C}^{*\{1,\ldots,n\}}_{u_{i_1},\ldots, u_{i_m}}$ as a mixed polynomial function of the $n-m$ variables $z_{i_{m+1}},\ldots,z_{i_n}$. (Here, $\{i_{m+1},\ldots,i_n\}=\{1,\ldots, n\}\setminus \{i_{1},\ldots,i_m\}$.) 

We say that $f$ is \emph{locally tame along the vanishing coordinates subspace $\mathbb{C}^I$}, $I\in\mathscr{I}_v(f)$, if for any essential non-compact face $\Delta\subseteq\Gamma_{nc}(f;\mathbf{z},\bar{\mathbf{z}})$ with $I_{\Delta}=I$, the face function $f_\Delta$ is locally tame. Finally, we say that $f$ is \emph{locally tame along its vanishing coordinates subspaces} if it is locally tame along $\mathbb{C}^I$ for any $I\in\mathscr{I}_v(f)$.
\end{definition}

\begin{notation}\label{notationrayon}
For any given essential non-compact face $\Delta$, we denote by $\bar{r}(f_\Delta)$ the supremum of the numbers $r(f_\Delta)$ satisfying the inequality (\ref{irlt}).
Then, for any $I\in\mathscr{I}_v(f)$, we set
\begin{equation*}
r_I(f):=\mbox{min} \{\bar{r}(f_\Delta)\mid \Delta \mbox{ is an essential non-compact face with } I_\Delta=I\}.
\end{equation*} 
Finally, we put
\begin{equation*}
r_{nc}(f) := \mbox{min} \{ r_I(f) \mid I\in\mathscr{I}_v(f)\}.
\end{equation*}
The number $r_{nc}(f)$ is called the \emph{radius of local tameness} of $f$.
\end{notation}

\begin{example}
If $f(z_1,z_2,\bar{z}_1,\bar{z}_2)=\bar{z}_1^2z_2^3+z_1^3\bar{z}_2^2+2z_1^2z_2^4$, then $\Gamma_{nc}(f;\mathbf{z},\bar{\mathbf{z}})$ has two essential non-compact faces $\Delta_1:=A+\mathbb{R}_+\mathbf{e}_2$ and $\Delta_2:=B+\mathbb{R}_+\mathbf{e}_1$, where $A=(2,3)$ and $B=(3,2)$. Here, $I_{\Delta_1}=\{2\}$ and $I_{\Delta_2}=\{1\}$. For any $u_2\in\mathbb{C}^*$ with $\vert u_2\vert<1/2$, the mixed function 
\begin{equation*}
z_1\mapsto f_{\Delta_1}(z_1,u_2,\bar{z}_1,\bar{u}_2)=\bar{z}_1^2u_2^3+2z_1^2u_2^4
\end{equation*}
of the variable $z_1$ has no critical point on $\mathbb{C}^{*\{1,2\}}_{u_2}$. Thus the face function $f_{\Delta_1}$ is locally tame (we can take $r(f_{\Delta_1})=1/2$). Similarly, for any $u_1\in\mathbb{C}^*$, the mixed function 
\begin{equation*}
z_2\mapsto f_{\Delta_2}(u_1,z_2,\bar{u}_1,\bar{z}_2)=u_1^3\bar{z}_2^2
\end{equation*}
of the variable $z_2$ has no critical point on $\mathbb{C}^{*\{1,2\}}_{u_1}$, and hence the face function $f_{\Delta_2}$ is locally tame. Altogether, $f$ is locally tame along its vanishing coordinates subspaces.
\end{example}

\section{Admissible families and main theorem}\label{sect-ult}

Now consider a complex-valued polynomial function 
\begin{equation*}
f\colon \mathbb{C}\times \mathbb{C}^n\to \mathbb{C},\
(t,\mathbf{z})\mapsto f(t,\mathbf{z},\bar{\mathbf{z}}),
\end{equation*}
of the variables $t$, $\mathbf{z}:=(z_1,\ldots, z_n)$ and $\bar{\mathbf{z}}:=(\bar{z}_1,\ldots,\bar{z}_n)$ ($f$ does not depend on $\bar t$), and look at the $1$-parameter deformation family $\{f_t\}$ of mixed polynomial functions 
\begin{equation*}
f_t\colon \mathbb{C}^n\to \mathbb{C}, \ \mathbf{z}\mapsto f_t(\mathbf{z},\bar{\mathbf{z}}):=f(t,\mathbf{z},\bar{\mathbf{z}}).
\end{equation*}
As usual, we assume that for all $t$ sufficiently small, the origin of $\mathbb{C}^n$ belongs to the mixed hypersurface $V(f_t)$ and $f_t$ is not identically zero near $\mathbf{0}\in\mathbb{C}^n$.
In \S \ref{sub-admfam}, we introduce a condition (admissibility condition) that will ensure Whitney equisingularity for families of (possibly non-isolated) mixed singularities. Roughly, a family $\{f_t\}$ of mixed polynomial functions is \emph{admissible} if for all $t$ small enough, the non-compact Newton boundary $\Gamma_{nc}(f_t;\mathbf{z},\bar{\mathbf{z}})$ is independent of $t$, the mixed polynomial function $f_t$ is strongly non-degenerate and locally tame, and the radius of local tameness $r_{nc}(f_t)$ is greater than or equal to a fixed positive number $\rho>0$ (i.e., the family $\{f_t\}$ is \emph{uniformly} locally tame). This is a ``mixed'' version of the admissibility condition introduced in \cite{EO} in the special case of holomorphic polynomial functions.  Our main result (\S \ref{sub-admfam}, Theorem \ref{mt2}) says that if a family $\{f_t\}$ satisfies the admissibility condition, then the corresponding family of mixed hypersurfaces $\{V(f_t)\}$ is Whitney  equisingular and satisfies the Thom $a_f$ condition. In particular, $\{V(f_t)\}$ is topologically equisingular. See \S \ref{defwrthms} for the definitions. The proof will show that if, in addition to the above assumptions, the function $f_t$ is also convenient---in which case $f_t$ has an isolated mixed singularity at $\mathbf{0}\in\mathbb{C}^n$---then we can drop the uniform local tameness assumption.

Before going into further details, we first show two important preliminary propositions. The first one says that if the Newton boundary $\Gamma(f_t;\mathbf{z},\bar{\mathbf{z}})$ is independent of $t$ and $f_t$ is strongly non-degenerate for all small $t$, then, in a neighbourhood of the origin $\mathbf{0}\in\mathbb{C}^n$, the mixed hypersurface $V(f_t)$ is mixed non-singular along $\mathbb{C}^{*I}$ for any $I\in\mathscr{I}_{nv}(f_t)$ and any $t$ small enough (cf.~\S \ref{sub-unismooth}, Proposition \ref{lemmasmooth}). The second proposition asserts that if, in addition to the above assumptions, $f_t$ is locally tame along the vanishing coordinates subspaces, then, in a neighbourhood of the origin, the nearby fibres $f_t^{-1}(\eta)$ are mixed non-singular for any $t$ and any $\eta\not=0$ small enough (cf.~\S \ref{sub-unitrans}, Proposition \ref{lemma-transversality}).

\subsection{Mixed smoothness along the non-vanishing coordinates subspaces}\label{sub-unismooth}

The following proposition is both a ``uniform'' version of \cite[Theorem 19]{O3} and a ``mixed'' version of \cite[Proposition 3.1]{EO}.

\begin{proposition}\label{lemmasmooth}
Suppose that for all $t$ sufficiently small, the following two conditions are satisfied:
\begin{enumerate}
\item
the Newton boundary $\Gamma(f_t;\mathbf{z},\bar{\mathbf{z}})$ of $f_t$ at the origin is independent of $t$ (in particular, $\mathscr{I}_{nv}(f_t)$ is independent of $t$);
\item
the mixed polynomial function $f_t$ is strongly non-degenerate. 
\end{enumerate}
Then there exists a positive number $R>0$ such that for any $I\in\mathscr{I}_{nv}(f_0)$ and any $t$ sufficiently small, $V(f_t)\cap \mathbb{C}^{*I}\cap B_R$ is mixed non-singular and intersects transversely with $S_r$ for any $r<R$, where $B_R$ (respectively, $S_r$) is the open ball (respectively, the sphere) with centre the origin $\mathbf{0}\in\mathbb{C}^n$ and radius $R$ (respectively, $r$).
\end{proposition}

\begin{proof}
As there are only finitely many subsets $I\in\mathscr{I}_{nv}(f_0)$, it suffices to show that for a fixed $I\in\mathscr{I}_{nv}(f_0)$, there is $R>0$ such that for any $t$ small enough, $V(f_t)\cap \mathbb{C}^{*I}\cap B_R$ is mixed non-singular and intersects transversely with the sphere $S_r$ for any $r\leq R$. To simplify, we assume that $I=\{1,\ldots,m\}$. (The argument in the other cases is similar.)

We start with the ``mixed smoothness'' assertion. We argue by contradiction. Suppose that there exists a sequence $\{(t_N,\mathbf{z}_N)\}$ of points in $V(f)\cap(\mathbb{C}\times\mathbb{C}^{*I})$ converging to $(0,\mathbf{0})$ and such that $\mathbf{z}_N$ is a critical point of the restriction of the mixed polynomial function $f_{t_N}$ to $\mathbb{C}^{I}$. 
Then $(0,\mathbf{0})$ is in the closure of the algebraic set $W$ consisting of the points $(t,\mathbf{z})\in \mathbb{C}\times\mathbb{C}^{*I}$ satisfying the following two conditions: 
\begin{enumerate}
\item
${f_t}_{\mid \mathbb{C}^{I}}(\mathbf{z},\bar{\mathbf{z}})=0$;
\item
there exists $\lambda_{t,\mathbf{z}}\in\mathbb{S}^1$ such that $\overline{\partial {f_t}_{\mid \mathbb{C}^{I}}}(\mathbf{z},\bar{\mathbf{z}})=\lambda_{t,\mathbf{z}}\, 
\overline{\partial} {f_t}_{\mid \mathbb{C}^{I}}(\mathbf{z},\bar{\mathbf{z}})$.
\end{enumerate}
Therefore, by the curve selection lemma (cf.~\cite{Milnor}), there is a real analytic curve 
\begin{align*}
(t(s),\mathbf{z}(s))=(t(s),z_1(s),\ldots,z_m(s),0,\ldots,0) 
\end{align*}
and a Laurent series $\lambda(s)\in\mathbb{S}^1$ such that:
\begin{enumerate}
\item[(i)]
$(t(0),\mathbf{z}(0))=(0,\mathbf{0})$;
\item[(ii)]
$(t(s),\mathbf{z}(s))\in \mathbb{C}\times \mathbb{C}^{*I}$ for $s\not=0$;
\item [(iii)]
$f_{t(s)}(\mathbf{z}(s),\bar{\mathbf{z}}(s))=0$;
\item [(iv)]
$\overline{\partial {f_{t(s)}}_{\mid \mathbb{C}^{I}}}(\mathbf{z}(s),\bar{\mathbf{z}}(s))=\lambda(s)\, \overline{\partial} {f_{t(s)}}_{\mid \mathbb{C}^{I}}(\mathbf{z}(s),\bar{\mathbf{z}}(s))$ for $s\not=0$;
\end{enumerate}
where $\bar{\mathbf{z}}(s):=(\bar{z}_1(s),\ldots,\bar{z}_n(s))$ and $\bar{z}_i(s)$ is the complex conjugate of $z_i(s)$.
Consider the Taylor expansions 
\begin{align*}
t(s)=t_0s^{v}+\cdots
\quad\mbox{and}\quad 
z_{i}(s)=a_{i}\, s^{w_{i}}+\cdots \ (1\leq i\leq m),
\end{align*}
where $t_0,\, a_{i}\not=0$ and $v,\, w_{i}>0$, and the Laurent expansion
\begin{align*}
\lambda(s)=\lambda_0+\lambda_1 s+\cdots,
\end{align*}
where $\lambda_0\in\mathbb{S}^1\subseteq\mathbb{C}$.
Throughout the dots stand for the higher order terms.
Let $\mathbf{a}:=(a_1,\ldots,a_m,0,\ldots,0)\in\mathbb{C}^{*I}$ and $\mathbf{w}:=(w_1,\ldots,w_m,0,\ldots,0)\in\mathbb{N}^{*I}$, and let $\Delta_\mathbf{w}$ be the  (compact) face of $\Gamma\bigl({f_{t(s)}}_{\mid \mathbb{C}^{I}};\mathbf{z},\bar{\mathbf{z}}\bigr)=\Gamma\bigl({f_0}_{\mid \mathbb{C}^{I}};\mathbf{z},\bar{\mathbf{z}}\bigr)$  defined by the locus
where the map
\begin{align*}
\xi:=(\xi_1,\ldots, \xi_m,0,\ldots,0)\in \Gamma({f_{t(s)}}_{\mid \mathbb{C}^{I}};\mathbf{z},\bar{\mathbf{z}}) \mapsto \sum_{1\leq i\leq m} \xi_i w_i
\end{align*}
takes its minimal value $d_\mathbf{w}$. 
For any $1\leq i\leq m$,
\begin{equation}\label{lemmasmooth-eq1}
\begin{aligned}
\frac{\partial \bigl({f_{t(s)}}_{\mid \mathbb{C}^{I}}\bigr)}
{\partial z_{i}}(\mathbf{z}(s),\bar{\mathbf{z}}(s)) = 
\frac{\partial \bigl({f_{t(s)}}_{\mid \mathbb{C}^{I}}\bigr)_{\Delta_\mathbf{w}}}
{\partial z_{i}}(\mathbf{a},\bar{\mathbf{a}})\, s^{d_\mathbf{w}-w_{i}}+\cdots,\\
\frac{\partial \bigl({f_{t(s)}}_{\mid \mathbb{C}^{I}}\bigr)}
{\partial \bar{z}_{i}}(\mathbf{z}(s),\bar{\mathbf{z}}(s)) = 
\frac{\partial \bigl({f_{t(s)}}_{\mid \mathbb{C}^{I}}\bigr)_{\Delta_\mathbf{w}}}
{\partial \bar{z}_{i}}(\mathbf{a},\bar{\mathbf{a}})\, s^{d_\mathbf{w}-w_{i}}+\cdots,
\end{aligned}
\end{equation}
where $\bigl({f_{t(s)}}_{\mid \mathbb{C}^{I}}\bigr)_{\Delta_\mathbf{w}}$ is the face function associated with ${f_{t(s)}}_{\mid \mathbb{C}^{I}}$ and $\Delta_\mathbf{w}$. The relation (iv)  says that for any $1\leq i\leq m$ and any $s\not=0$,
\begin{align*}
\frac{\overline{\partial \bigl({f_{t(s)}}_{\mid \mathbb{C}^{I}}\bigr)_{\Delta_\mathbf{w}}}} {\partial z_{i}}(\mathbf{a},\bar{\mathbf{a}})\, s^{d_\mathbf{w}-w_i} 
+ \cdots = \lambda_0 \frac{\partial \bigl({f_{t(s)}}_{\mid \mathbb{C}^{I}}\bigr)_{\Delta_\mathbf{w}}} {\partial \bar{z}_{i}}(\mathbf{a},\bar{\mathbf{a}})\, s^{d_\mathbf{w}-w_i} + \cdots .
\end{align*}
Dividing by $s^{d_{\mathbf{w}}-w_i}$ and taking the limit as $s\to 0$ gives
\begin{align*}
\overline{\partial \bigl({f_{0}}_{\mid \mathbb{C}^{I}}\bigr)_{\Delta_\mathbf{w}}}(\mathbf{a},\bar{\mathbf{a}})=\lambda_0\, \overline{\partial} \bigl({f_{0}}_{\mid \mathbb{C}^{I}}\bigr)_{\Delta_\mathbf{w}}(\mathbf{a},\bar{\mathbf{a}}),
\end{align*}
that is, $\mathbf{a}\in\mathbb{C}^{*I}$ is a critical point of $\bigl({f_{0}}_{\mid \mathbb{C}^{I}}\bigr)_{\Delta_\mathbf{w}}\colon \mathbb{C}^{I}\to\mathbb{C}$. In particular, this implies that the mixed polynomial function ${f_{0}}_{\mid \mathbb{C}^{I}}$ is \emph{not} strongly non-degenerate as a mixed function of the variables $z_1,\ldots, z_m$. This contradicts Proposition 7 of \cite{O3} which says that if a mixed polynomial function ${f_{0}}(\mathbf{z},\bar{\mathbf{z}})$ is strongly non-degenerate and if ${f_{0}}_{\mid \mathbb{C}^{I}}\not\equiv 0$, then ${f_{0}}_{\mid \mathbb{C}^{I}}$ is strongly non-degenerate as a mixed function of the variables $z_i$, $i\in I$. 

To prove the ``transversality'' assertion, we use the following lemma.

\begin{lemma}[cf.~{\cite[Lemma 2]{O1} and \cite[Lemma 3.2.1]{C}}]\label{lemmaOkaChen}
Let  $k(\mathbf{z},\bar{\mathbf{z}})$ be a mixed polynomial function, let $\mathbf{p}$ be a mixed non-singular point of $V(k):=k^{-1}(0)$, and let $r(\mathbf{z},\bar{\mathbf{z}})$ be a real-valued mixed function. Then the following three assertions are equivalent:
\begin{enumerate}
\item
The restriction of $r$ to $V(k)$ has a critical point at $\mathbf{p}$;
\item
There exists a complex number $\lambda\in\mathbb{C}^*$ such that
\begin{align*}
\qquad
\bar{\partial} r(\mathbf{p},\bar{\mathbf{p}})=\lambda \overline{\partial k}(\mathbf{p},\bar{\mathbf{p}}) + \bar{\lambda} \bar{\partial}k(\mathbf{p},\bar{\mathbf{p}});
\end{align*}
\item
There exist real numbers $\alpha$ and $\beta$ such that
\begin{align*}
\qquad
\bar{\partial} r(\mathbf{p},\bar{\mathbf{p}})=\alpha\bar{\partial} k_1(\mathbf{p},\bar{\mathbf{p}}) + \beta \bar{\partial}k_2(\mathbf{p},\bar{\mathbf{p}}),
\end{align*}
where $k_1$ and $k_2$ are the real and imaginary parts of $k$ respectively.
\end{enumerate}
\end{lemma}

Using this lemma, we can prove the ``transversality'' assertion of Proposition \ref{lemmasmooth} as follows. Again, we argue by contradiction. Suppose that there exists a sequence $\{(t_N,\mathbf{z}_N)\}$ of points in $V(f)\cap(\mathbb{C}\times \mathbb{C}^{*I})$ converging to $(0,\mathbf{0})$ and such that $V(f_{t_N})\cap \mathbb{C}^{*I}$ does not intersect the sphere $S_{\Vert z_N\Vert}$ transversely at $\mathbf{z}_N$. Then, by Lemma \ref{lemmaOkaChen} applied with the squared distance function $r(\mathbf{z},\bar{\mathbf{z}}):=\Vert \mathbf{z}\Vert^2=\sum_{i=1}^n (x_i^2+y_i^2)$, the origin $(0,\mathbf{0})$ of $\mathbb{C}\times\mathbb{C}^n$ is in the closure of the set consisting of the points $(t,\mathbf{z})\in \mathbb{C}\times \mathbb{C}^{*I}$ satisfying the following two conditions:
\begin{enumerate}
\item
${f_t}_{\mid \mathbb{C}^{I}}(\mathbf{z},\bar{\mathbf{z}})=0$;
\item
there exists real numbers 
$\alpha_{t,\mathbf{z}}$ and $\beta_{t,\mathbf{z}}$ such that
\begin{align*}
\qquad
\mathbf{z}=\alpha_{t,\mathbf{z}}\, \bar{\partial} {g_{t,\bar t}}_{\mid \mathbb{C}^I}(\mathbf{z},\bar{\mathbf{z}}) + \beta_{t,\mathbf{z}}\, \bar{\partial}{h_{t,\bar t}}_{\mid \mathbb{C}^I}(\mathbf{z},\bar{\mathbf{z}}),
\end{align*}
where $g_{t,\bar t}:=\Re(f_t)$ and $h_{t,\bar t}:=\Im(f_t)$ are the real and imaginary parts of $f_t$ respectively. (Note that, unlike the function $f$ which does not depend on $\bar t$, the functions $g$ and $h$ do depend on it.)
\end{enumerate}
Thus, by the curve selection lemma, we can find a real analytic curve 
\begin{align*}
(t(s),\mathbf{z}(s))=(t(s),z_1(s),\ldots,z_m(s),0,\ldots,0) 
\end{align*}
and Laurent series $\alpha(s)$ and $\beta(s)$ such that:
\begin{enumerate}
\item[(i)]
$(t(0),\mathbf{z}(0))=(0,\mathbf{0})$;
\item[(ii)]
$(t(s),\mathbf{z}(s))\in \mathbb{C}\times \mathbb{C}^{*I}$ for $s\not=0$;
\item [(iii)]
$f_{t(s)}(\mathbf{z}(s),\bar{\mathbf{z}}(s))=0$;
\item [(iv)]
$\mathbf{z}(s)=\alpha(s)\, \bar{\partial} {g_{t(s),\bar t(s)}}_{\mid \mathbb{C}^I}(\mathbf{z}(s),\bar{\mathbf{z}}(s)) + \beta(s)\, \bar{\partial}{h_{t(s),\bar t(s)}}_{\mid \mathbb{C}^I}(\mathbf{z}(s),\bar{\mathbf{z}}(s))$ for $s\not=0$.
\end{enumerate}
Consider the Taylor expansions 
\begin{align*}
t(s)=t_0s^{v}+\cdots
\quad\mbox{and}\quad
z_{i}(s)=a_{i}\, s^{w_{i}}+\cdots \ (1\leq i\leq m),
\end{align*}
where $t_0,\, a_{i}\not=0$ and $v,\, w_{i}>0$, 
and the Laurent expansions
\begin{align*}
\alpha(s)=\alpha_0\, s^\omega+\cdots
\quad\mbox{and}\quad
\beta(s)=\beta_0\, s^{\omega'}+\cdots,
\end{align*}
where $\alpha_0,\, \beta_0\not=0$.
Then define $\mathbf{a}$, $\mathbf{w}$, $d_{\mathbf{w}}$ and $\Delta_{\mathbf{w}}$ as above. 
The condition (iv) says that for any $1\leq i\leq m$ and any $s\not=0$,
\begin{equation}\label{cond-de}
\begin{aligned}
a_{i}\, s^{w_{i}}+\cdots =
 \biggl(\alpha_0\, & \frac{\partial \bigl({g_{t(s),\bar t(s)}}_{\mid \mathbb{C}^I}\bigr)_{\Delta_{\mathbf{w}}}}{\partial \bar{z}_i} 
 (\mathbf{a},\bar{\mathbf{a}}) \, s^{\omega+d_{\mathbf{w}}-w_i} +\cdots\biggr) + \\
& \biggl(\beta_0\, \frac{\partial \bigl({h_{t(s),\bar t(s)}}_{\mid \mathbb{C}^I}\bigr)_{\Delta_{\mathbf{w}}}}{\partial \bar{z}_i}(\mathbf{a},\bar{\mathbf{a}})\, s^{\omega'+d_{\mathbf{w}}-w_i} +\cdots\biggr). 
\end{aligned}
\end{equation}
By reordering, we may assume that $w_1=\cdots=w_k<w_j$ ($k<j\leq m$).
To simplify, let us explain the argument in the case where $\omega'=\omega$. (The argument in the other cases is similar.)
If $\omega+d_{\mathbf{w}}-w_i<w_i$ for all $1\leq i\leq m$, then 
\begin{align*}
\alpha_0\, \frac{\partial \bigl({g_{0}}_{\mid \mathbb{C}^{I}}\bigr)_{\Delta_{\mathbf{w}}}} {\partial \bar{z}_{i}}(\mathbf{a},\bar{\mathbf{a}}) + \beta_0\, \frac{\partial \bigl({h_{0}}_{\mid \mathbb{C}^{I}}\bigr)_{\Delta_{\mathbf{w}}}} {\partial \bar{z}_{i}}(\mathbf{a},\bar{\mathbf{a}}) = 0
\end{align*}
for all $1\leq i\leq m$, where $g_0:=g_{0,\bar 0}=\Re(f_0)$ and $h_0:=h_{0,\bar 0}=\Im(f_0)$.
Therefore $\mathbf{a}\in\mathbb{C}^{*I}$ is a critical point of $({f_{0}}_{\mid \mathbb{C}^{I}})_{\Delta_{\mathbf{w}}}$, which is a contradiction with the strong non-degeneracy of ${f_{0}}_{\mid \mathbb{C}^{I}}$. Now, if there exists $i_0$, $1\leq i_0\leq m$, such that $\omega+d_{\mathbf{w}}-w_{i_0}=w_{i_0}$, then (\ref{cond-de}) shows that $1\leq i_0\leq k$, $\omega+d_{\mathbf{w}}-w_1=w_1$, and
\begin{align}\label{exprder}
\alpha_0\, \frac{\partial \bigl({g_{0}}_{\mid \mathbb{C}^{I}}\bigr)_{\Delta_{\mathbf{w}}}} {\partial \bar{z}_{i}}(\mathbf{a},\bar{\mathbf{a}}) + \beta_0\, \frac{\partial \bigl({h_{0}}_{\mid \mathbb{C}^{I}}\bigr)_{\Delta_{\mathbf{w}}}} {\partial \bar{z}_{i}}(\mathbf{a},\bar{\mathbf{a}}) =
\left\{
\begin{aligned}
& a_i &&\mbox{for} && 1\leq i\leq k,\\
& 0 &&\mbox{for} && k< i\leq m.
\end{aligned}
\right.
\end{align}
Moreover, as $f(t(s),\mathbf{z}(s),\bar{\mathbf{z}}(s))=0$ for any $s$, we have
\begin{align*}
\bigl({g_{0}}_{\mid \mathbb{C}^{I}}\bigr)_{\Delta_{\mathbf{w}}}(\mathbf{a},\bar{\mathbf{a}}) = \bigl({h_{0}}_{\mid \mathbb{C}^{I}}\bigr)_{\Delta_{\mathbf{w}}}
(\mathbf{a},\bar{\mathbf{a}})=0.
\end{align*}
As face functions with positive weights are always \emph{radially weighted homogeneous} (cf.~\cite[\S 2]{O3}), the mixed polynomial function 
\begin{equation*}
\ell(\mathbf{z},\bar{\mathbf{z}}):=\alpha_0 ({g_{0}}_{\mid \mathbb{C}^{I}}\bigr)_{\Delta_{\mathbf{w}}}(\mathbf{z},\bar{\mathbf{z}}) + \beta_0 ({h_{0}}_{\mid \mathbb{C}^{I}}\bigr)_{\Delta_{\mathbf{w}}}(\mathbf{z},\bar{\mathbf{z}})
\end{equation*}
is radially weighted homogeneous of type $(w_1,\ldots,w_m;d_{\mathbf{w}})$. Therefore we have the following \emph{radial Euler identity}:
\begin{align*}
d_{\mathbf{w}}\cdot \ell(\mathbf{z},\bar{\mathbf{z}}) = \sum_{1\leq i\leq m} w_i\biggl( z_i \frac{\partial \ell}{\partial z_i} (\mathbf{z},\bar{\mathbf{z}}) + \bar{z}_i \frac{\partial \ell}{\partial \bar{z}_i} (\mathbf{z},\bar{\mathbf{z}}) \biggr)
\end{align*}
for any $\mathbf{z}\in\mathbb{C}^n$.
Applying this identity with $\mathbf{z}=\mathbf{a}$ gives
\begin{align*}
0 & = d_{\mathbf{w}}\cdot \ell(\mathbf{a},\bar{\mathbf{a}}) = \sum_{1\leq i\leq m} w_i\biggl( a_i \frac{\partial \ell}{\partial z_i} (\mathbf{a},\bar{\mathbf{a}}) + \bar{a}_i \frac{\partial \ell}{\partial \bar{z}_i} (\mathbf{a},\bar{\mathbf{a}}) \biggr) \\
& = 2\Re\biggl( \sum_{1\leq i\leq m} w_i \bar{a}_i \frac{\partial \ell}{\partial \bar{z}_i} (\mathbf{a},\bar{\mathbf{a}}) \biggr)
\overset{\mbox{\tiny by (\ref{exprder})}}{=} 
2 w_1 \sum_{1\leq i\leq k}  \bar{a}_i a_i
  = 2w_1 \sum_{1\leq i\leq k} \vert a_i\vert^2 > 0,
\end{align*}
which is a contradiction.
This completes the proof of Proposition \ref{lemmasmooth}.
\end{proof}

\begin{remark}
Actually, the proof shows that we do not need to assume that $f_t$ is strongly non-degenerate for all $t$. It is enough to make this assumption for $t=0$.
\end{remark}

\subsection{Mixed smoothness of the nearby fibres}\label{sub-unitrans}

The following proposition generalizes \cite[Lemmas 4 and 7]{O1} and \cite[Lemma 28]{O3} to $1$-parameter deformation families. Though this proposition is not necessary for the purpose of our paper, it is interesting in itself and may be useful for other purposes. For this reason, we think it is worth to include it here.

\begin{proposition}\label{lemma-transversality}
Suppose that for all $t$ sufficiently small, the following two conditions are satisfied:
\begin{enumerate}
\item
the Newton boundary $\Gamma(f_t;\mathbf{z},\bar{\mathbf{z}})$ of $f_t$ at the origin is independent of $t$;
\item
the mixed polynomial function $f_t$ is strongly non-degenerate and locally tame along the vanishing coordinates subspaces. 
\end{enumerate}
Then there exists a positive number $R'>0$ such that for any $0<R''\leq R'$, there exists $\delta(R'')>0$ such that  for any $\eta\not=0$ with $\vert\eta\vert\leq\delta(R'')$ and any $r$ with $R''\leq r\leq R'$, the set $f_t^{-1}(\eta)\cap B_{R'}$ is mixed non-singular and transversely intersects the sphere $S_{r}$ for all $t$ small enough. 
\end{proposition}

Again, $B_{R'}$ (respectively, $S_r$) is the open ball (respectively, the sphere) with centre the origin $\mathbf{0}\in\mathbb{C}^n$ and radius $R'$ (respectively, $r$).

\begin{proof}
We shall omit the proof of the mixed smoothness of $f_t^{-1}(\eta)\cap B_{R'}$, which can be shown by an argument similar to that used in the proof of  Proposition \ref{lemmasmooth}. 
In particular, the only requirements for this part of the proposition is the independence of the Newton boundary with respect to the parameter $t$ and the strong non-degeneracy of $f_t$.

The transversality to the spheres is more subtle, as it requires (in addition to the above conditions) the local tameness of $f_t$. For this reason, we shall give the full details for this part. Again, we argue by contradiction. If the assertion is false, then,
by the curve selection lemma and by Lemma \ref{lemmaOkaChen}, we can find a real analytic curve $(t(s),\mathbf{z}(s))$
and a Laurent series $\lambda(s)$ such that:
\begin{enumerate}
\item[(i)]
$f_{t(s)}(\mathbf{z}(s),\bar{\mathbf{z}}(s))\not=0$ for $s\not=0$ and $f_0(\mathbf{z}(0),\bar{\mathbf{z}}(0))=0$;
\item [(ii)]
$\mathbf{z}(s)=\lambda(s)\, \overline{{\partial} {f_{t(s)}}}   
(\mathbf{z}(s),\bar{\mathbf{z}}(s)) + \bar{\lambda}(s)\, 
\bar{\partial}{f_{t(s)}} (\mathbf{z}(s),\bar{\mathbf{z}}(s))$ for $s\not=0$;
\end{enumerate}
where $\bar{\lambda}(s)$ is the complex conjugate of $\lambda(s)$.
Consider the Laurent expansion
\begin{align*}
\lambda(s)=\lambda_0\, s^\omega+\cdots,
\end{align*}
where $\lambda_0\not=0$, and the Taylor expansions 
\begin{align*}
t(s)=t_0s^{v}+\cdots
\quad\mbox{and}\quad
z_{i}(s)=a_{i}\, s^{w_{i}}+\cdots,
\end{align*}
where $t_0\not=0$, $v>0$, and $a_{i}\not=0$ if the function $z_i(s)$ is not identically zero.
Let $K:=\{i\in\{1,\ldots,n\}\mid z_i(s)\not\equiv 0\}$. To simplify, we assume that $K=\{1,\ldots,n\}$. (The argument is exactly the same in the other cases.)
Then we define $\mathbf{a}:=(a_1,\ldots,a_n)$, $\mathbf{w}:=(w_1,\ldots,w_n)$, and we consider the face $\Delta_\mathbf{w}$ of the Newton polygon $\Gamma_+\bigl({f_{t(s)}};\mathbf{z},\bar{\mathbf{z}}\bigr)=\Gamma_+\bigl({f_0};\mathbf{z},\bar{\mathbf{z}}\bigr)$  defined by the locus where the map
\begin{align*}
\xi:=(\xi_1,\ldots, \xi_n)\in \Gamma_+({f_{t(s)}};\mathbf{z},\bar{\mathbf{z}}) \mapsto \sum_{1\leq i\leq n} \xi_i w_i
\end{align*}
takes its minimal value $d_\mathbf{w}$. Finally, we set $I:=\{i\in K\mid w_i=0\}$. 
 The condition (ii) says that for any $i\in K$ and any $s\not=0$,
\begin{align*}
a_{i}\, s^{w_{i}}+\cdots =
 \biggl(\lambda_0\, \overline{\frac{\partial \bigl({f_{t(s)}}\bigr)_{\Delta_\mathbf{w}}}{\partial z_i}} & (\mathbf{a},\bar{\mathbf{a}}) \, s^{\omega+d_\mathbf{w}-w_i} +\cdots\biggr) + \\
& \biggl(\bar{\lambda}_0\, \frac{\partial \bigl({f_{t(s)}}\bigr)_{\Delta_\mathbf{w}}}{\partial \bar{z}_i}(\mathbf{a},\bar{\mathbf{a}})\, s^{\omega+d_\mathbf{w}-w_i} +\cdots\biggr).
\end{align*}
Let us first assume that $I$ is not empty and belongs to $\mathscr{I}_v(f_0)=\mathscr{I}_v(f_t)$.
The lowest power $p_i$ of the left-hand side of the above equation is $w_i$ (in particular, if $i\in I$, then $p_i=0$) while the lowest (possibly negative) power $q_i$ of the right-hand side is greater than or equal to $\omega+d_\mathbf{w}-w_i$ (if $i\in I$, then $q_i\geq\omega+d_\mathbf{w}$). Clearly, $\omega+d_\mathbf{w}\leq 0$. (Otherwise, comparing both sides for $i\in I\not=\emptyset$ gives $0<\omega+d_\mathbf{w}\leq q_i=p_i=0$, which is a contradiction.) 
If $\omega+d_\mathbf{w}<0$, then for all $1\leq i\leq n$,
\begin{align*}
0=\lambda_0\, \overline{\frac{\partial \bigl({f_{0}}\bigr)_{\Delta_\mathbf{w}}}{\partial z_i}}(\mathbf{a},\bar{\mathbf{a}}) + \bar{\lambda}_0\, \frac{\partial \bigl({f_{0}}\bigr)_{\Delta_\mathbf{w}}}{\partial \bar{z}_i}(\mathbf{a},\bar{\mathbf{a}}),
\end{align*}
while if $\omega+d_\mathbf{w}=0$, then 
\begin{align*}
\lambda_0\, \overline{\frac{\partial \bigl({f_{0}}\bigr)_{\Delta_\mathbf{w}}}{\partial z_i}}(\mathbf{a},\bar{\mathbf{a}}) + \bar{\lambda}_0\, \frac{\partial \bigl({f_{0}}\bigr)_{\Delta_\mathbf{w}}}{\partial \bar{z}_i}(\mathbf{a},\bar{\mathbf{a}}) =
\left\{
\begin{aligned}
& a_i &&\mbox{for} && i\in I,\\
& 0 &&\mbox{for} && i\notin I.
\end{aligned}
\right.
\end{align*}
The above equations show that if $\omega+d_\mathbf{w}<0$, then the point $\mathbf{a}$ is a critical point of $({f_{0}})_{\Delta_\mathbf{w}}$, while if $\omega+d_\mathbf{w}=0$, then $\mathbf{a}$ is a critical point of $({f_{0}})_{\Delta_\mathbf{w}}$ as a mixed function of the variables $z_i$ with $i\notin I$, fixing $z_i=a_i$ for $i\in I$. In both cases, we get a contradiction with the local tameness of $({f_{0}})_{\Delta_\mathbf{w}}$. (Note that by considering the function $s\mapsto z_i(s/b_i)$ for a sufficiently big $b_i$, if necessary, we may always assume that $\sum\vert a_i\vert^2\leq \mbox{min}\{r_{nc}(f_0),R\}$, where $r_{nc}(f_0)$ is the radius of local tameness of $f_0$ and $R$ is the positive number which appears in Proposition \ref{lemmasmooth}.) 

Now, if $I$ is empty, then an argument similar to that given in Proposition \ref{lemmasmooth} leads to a contradiction with the strong non-degeneracy of $f_0$ and the radial Euler identity. (In particular, in this case, the local tameness assumption for $f_0$ can be dropped.) Details are straightforward and left to~the~reader.

Finally, if $I\in\mathscr{I}_{nv}(f_0)$, then we get a contradiction with Proposition \ref{lemmasmooth}. Indeed, let $\mathbf{a}_I:=(a_{I,1},\ldots, a_{I,n})$, where $a_{I,i}$ is equal to $a_i$ if $i\in I$ and to zero otherwise. Then $\mathbf{a}_I$ belongs to the set $V(f_0)\cap\mathbb{C}^{*I}$, and by Proposition \ref{lemmasmooth}, this set is transverse at $\mathbf{a}_I$ to the sphere $S_{\Vert \mathbf{a}_I\Vert}$ provided that $\Vert \mathbf{a}_I\Vert<R$. (The latter condition is always possible by considering $s\mapsto z_i(s/b_i)$ for $b_i$ large enough as above.) By stability of transversality, it follows that 
\begin{equation*}
V(f_{t(s)}-f_{t(s)}(\mathbf{z}(s))):=f_{t(s)}^{-1}(f_{t(s)}(\mathbf{z}(s)))
\end{equation*}
is transverse to the sphere $S_{\Vert \mathbf{z}(s)\Vert}$ at the point $\mathbf{z}(s)$ provided that $s$ is sufficiently small. This completes the proof of the proposition.
\end{proof}

\begin{remark}
Again, the proof shows that we do not need to assume that $f_t$ is strongly non-degenerate and locally tame for all $t$. Making this assumption for $t=0$ is enough.
\end{remark}

\begin{remark}\label{rkfisosing}
In the special case where, in addition to the assumptions of Proposition \ref{lemma-transversality}, the functions $f_t$ are \emph{convenient}, the only vanishing coordinates subspace is~$\mathbb{C}^{\emptyset}$. Thus, as observe in the above proof, in this case, the conclusions of the proposition still hold true without the local tameness assumption.
\end{remark}

\subsection{Whitney equisingularity and Thom condition}\label{defwrthms}

A (real) stratification of a subset $E$ of $\mathbb{R}^N$ (i.e., a partition of $E$ into real smooth submanifolds of $\mathbb{R}^N$) is called \emph{Whitney $(b)$-regular}  if for any pair of strata $(S_1,S_2)$ and any point $\mathbf{p}\in S_1\cap\bar S_2$ (where $\bar S_2$ is the closure of $S_2$ in $\mathbb{R}^N$), the stratum $S_2$ is Whitney $(b)$-regular over the stratum $S_1$ at the point~$\mathbf{p}$. The latter condition means that  
for any sequences of points $\{\mathbf{p}_k\}$ in $S_{1}$, $\{\mathbf{q}_k\}$ in $S_{2}$ and $\{a_k\}$ in $\mathbb{C}$ satisfying:
\begin{enumerate}
\item[(i)]
$\mathbf{p}_k\to\mathbf{p}$ and $\mathbf{q}_k\to\mathbf{p}$;
\item[(ii)]
$T_{\mathbf{q}_k} S_{2}\to T$;
\item[(iii)]
$a_k(\mathbf{p}_k-\mathbf{q}_k)\to v$;
\end{enumerate}
we have $v\in T$. (As usual, $T_{\mathbf{q}_k} S_{2}$ is the tangent space to $S_{2}$ at $\mathbf{q}_k$.) For details, we refer the reader to \cite{GWPL}.

\begin{remark}\label{remarkstratification}
Note that if $M$ is a real smooth manifold and $N\subseteq M$ is a real closed smooth submanifold, then $M\setminus N$ is Whitney $(b)$-regular over $N$ at any point.
\end{remark}

The following proposition, stated in \cite{O2}, says that in order to show that the Whitney $(b)$-regularity holds, it suffices to check it along real analytic paths. This is a consequence of the curve selection lemma (cf.~\cite{Milnor}) and Theorem 17.5 of \cite{Whitney}.

\begin{proposition}[cf.~{\cite[Proposition (2.2)]{O2}}]\label{prop-suitearc1}
If $\{\mathbf{p}_k\}$, $\{\mathbf{q}_k\}$ and $\{a_k\}$ are sequences satisfying the above conditions (i), (ii) and (iii), then there exist real analytic paths $\mathbf{p}(s)$, $\mathbf{q}(s)$ and $a(s)$ such that:
\begin{enumerate}
\item[(a)]
$\mathbf{p}(s)\in S_1$ and $a(s)\in\mathbb{C}$ for all $s$, and $\mathbf{q}(s)\in S_2$ for all $s\not=0$;
\item[(b)]
$\mathbf{p}(0)=\mathbf{q}(0)=\mathbf{p}$;
\item[(c)]
$T_{\mathbf{q}(s)} S_{2}\to T$;
\item[(d)]
$a(s)(\mathbf{p}(s)-\mathbf{q}(s))\to v$.
\end{enumerate}
\end{proposition}

Let $u:=\Re(t)$ and $v:=\Im(t)$ be the real and imaginary parts of $t$ respectively. Then the mixed hypersurface 
\begin{equation*}
V(f):=\{(t,\mathbf{z}) \in\mathbb{C}\times\mathbb{C}^n \mid f(t,\mathbf{z},\bar{\mathbf{z}})=0\}
\end{equation*}
may be understood as the real algebraic variety in $\mathbb{R}^2\times\mathbb{R}^{2n}$ defined by 
\begin{equation*}
g((u,v),\mathbf{x},\mathbf{y})=h((u,v),\mathbf{x},\mathbf{y})=0,
\end{equation*}
where $g:=\Re(f)$ and $h:=\Im(f)$ are the real and imaginary parts of $f$ respectively.

\begin{definition}
We say that the family of mixed hypersurfaces $\{V(f_t)\}$ is \emph{Whitney equisingular} if the exists a real Whitney $(b)$-regular stratification $\mathscr{S}$ of the mixed hypersurface $V(f)$ (understood as a real algebraic variety in $\mathbb{R}^2\times\mathbb{R}^{2n}$) such that the $t$-axis $\mathbb{C}\times \{\mathbf{0}\}$ (identified to $\mathbb{R}^2\times \{\mathbf{0}\}\subseteq\mathbb{R}^2\times\mathbb{R}^{2n}$) is a stratum.
\end{definition}

\begin{remark}
If $(V(f),\mathscr{S})$ is a real Whitney $(b)$-regular stratified set, then, by the Thom-Mather first isotopy theorem, $(V(f),\mathscr{S})$ is topologically locally trivial (see, e.g., Theorem (5.2) and Corollary (5.5) of \cite{GWPL}). Therefore, if, furthermore, the $t$-axis is a stratum (i.e., if the family $\{V(f_t)\}$ is Whitney equisingular), then $\{V(f_t)\}$ is topologically equisingular. The latter condition means that for all sufficiently small $t$, there exists an open neighbourhood $U_t$ of $\mathbf{0}\in\mathbb{C}^n$ together with a homeomorphism $\varphi_t\colon (U_t,\mathbf{0})\rightarrow (\varphi_t(U_t),\mathbf{0})$ such that $\varphi_t(V(f_0)\cap U_t)=V(f_t)\cap \varphi_t(U_t)$. In other words, the local ambient topological type of $V(f_t)$ at $\mathbf{0}\in\mathbb{C}
^n$ is independent of $t$ for all small $t$.
\end{remark}

Let $\Sigma f$ be the set of critical points of the mixed polynomial function $f$. Assume there exists a real Whitney $(a)$-regular\footnote{A (real) stratification of a subset $E\subseteq \mathbb{R}^N$ is called \emph{Whitney $(a)$-regular} if for any pair of strata $(S_1,S_2)$ and any point $\mathbf{p}\in S_1\cap\bar S_2$, the stratum $S_2$ is Whitney $(a)$-regular over the stratum $S_1$ at the point $\mathbf{p}$. That is, for any sequence of points $\{\mathbf{q}_k\}$ in $S_{2}$ satisfying 
\begin{equation*}
\mathbf{q}_k\to\mathbf{p}
\quad\mbox{and}\quad
T_{\mathbf{q}_k} S_{2}\to T, 
\end{equation*}
we have $T_{\mathbf{p}}S_1\subseteq T$. (Note that a Whitney $(b)$-regular stratification is always $(a)$-regular while the converse is not true.)} stratification $\mathscr{S}$ of $\mathbb{C}\times \mathbb{C}^n$ (identified with $\mathbb{R}^{2}\times \mathbb{R}^{2n}$) such that the mixed hypersurface $V(f)$ is a union of strata. We say that $\mathscr{S}$ satisfies \emph{Thom's $a_f$ condition} with respect to the complement of $\Sigma f$ if for any stratum $S\in\mathscr{S}$ contained in $V(f)$ and any point  $(\tau,\mathbf{q})\in S$, there exists an open neighbourhood $W$ of $(\tau,\mathbf{q})$ in $\mathbb{C}\times\mathbb{C}^n$ such that for any sequence of points $(\tau_k,\mathbf{q}_k)\in W\setminus\Sigma f$ satisfying
\begin{equation}\label{condtafc}
(\tau_k,\mathbf{q}_k)\to (\tau,\mathbf{q})
\quad\mbox{and}\quad
T_{(\tau_k,\mathbf{q}_k)}V(f-f(\tau_k,\mathbf{q}_k,\bar{\mathbf{q}}_k))\to T,
\end{equation}
we have $T_{(\tau,\mathbf{q})}S\subseteq T$, where
\begin{equation*}
V(f-f(\tau_k,\mathbf{q}_k,\bar{\mathbf{q}}_k)):=\{(t,\mathbf{z})\in \mathbb{C}\times\mathbb{C}^n \mid f(t,\mathbf{z},\bar{\mathbf{z}})=f(\tau_k,\mathbf{q}_k,\bar{\mathbf{q}}_k)\}.
\end{equation*}

Just as for the Whitney $(b)$-regularity, to show that the Thom condition holds, it suffices to check it along real analytic paths. (Again, this is a consequence of the curve selection lemma.) Precisely, we have the following statement.

\begin{proposition}\label{prop-suitearc2}
If $\{(\tau_k,\mathbf{q}_k)\}$ is a sequence of points in $W\setminus\Sigma f$ satisfying the above condition (\ref{condtafc}), then there exists a real analytic path $(\tau(s),\mathbf{q}(s))$ such that:
\begin{enumerate}
\item[(a)]
$(\tau(s),\mathbf{q}(s))\in W\setminus\Sigma f$ for all $s\not=0$;
\item[(b)]
$(\tau(0),\mathbf{q}(0))=(\tau,\mathbf{q})$;
\item[(c)]
$T_{(\tau(s),\mathbf{q}(s))}V(f-f(\tau(s),\mathbf{q}(s),\bar{\mathbf{q}}(s)))\to T$.
\end{enumerate}
\end{proposition}

\subsection{Admissibility condition and the main theorem}\label{sub-admfam}
Throughout this section, unless otherwise stated, we assume that for all $t$ sufficiently small the following two conditions hold:
\begin{enumerate}
\item[($\mbox{C}_1$)]
the non-compact Newton boundary $\Gamma_{nc}(f_t;\mathbf{z},\bar{\mathbf{z}})$ of $f_t$ at the origin is independent of $t$ (in particular, $\mathscr{I}_{nv}(f_t)$ and $\mathscr{I}_{v}(f_t)$ are independent of $t$);
\item[($\mbox{C}_2$)]
the mixed polynomial function $f_t$ is strongly non-degenerate and locally tame along the vanishing coordinates subspaces (i.e., locally tame along $\mathbb{C}^I$ for any $I\in\mathscr{I}_{v}(f_t)=\mathscr{I}_{v}(f_0)$); we denote by $r_{nc}(f_t)$ the radius of local tameness of $f_t$ (cf.~Definition \ref{definitionlocallytame} and Notation \ref{notationrayon}).
\end{enumerate}

\begin{remark}
Note that $\Gamma_{nc}(f_t;\mathbf{z},\bar{\mathbf{z}})=\Gamma_{nc}(f_0;\mathbf{z},\bar{\mathbf{z}})$ $\Leftrightarrow$ $\Gamma(f_t;\mathbf{z},\bar{\mathbf{z}})=\Gamma(f_0;\mathbf{z},\bar{\mathbf{z}})$ $\Leftrightarrow$ $\Gamma_{+}(f_t;\mathbf{z},\bar{\mathbf{z}})=\Gamma_{+}(f_0;\mathbf{z},\bar{\mathbf{z}})$.
\end{remark}

By Proposition \ref{lemmasmooth}, we know that there exists a positive number $R>0$ such that for any $I\in\mathscr{I}_{nv}(f_t)=\mathscr{I}_{nv}(f_0)$ and any $t$ small enough, $V(f_t)\cap \mathbb{C}^{*I}\cap B_R$ is mixed non-singular---and therefore it is a smooth real algebraic variety of real codimension two.
It follows immediately that in a sufficiently small open neighbourhood $U$ of the origin of $\mathbb{C}\times\mathbb{C}^n$, the set
$V(f)\cap (\mathbb{C}\times \mathbb{C}^{*I})$ is mixed non-singular for any $I\in\mathscr{I}_{nv}(f_t)$. Therefore, in such a neighbourhood, we can stratify $\mathbb{C}\times \mathbb{C}^n$ (identified with $\mathbb{R}^2\times\mathbb{R}^{2n}$) in such a way that the mixed hypersurface 
\begin{equation*}
V(f):=\{(t,\mathbf{z})\in\mathbb{C}\times\mathbb{C}^n\mid f(t,\mathbf{z},\bar{\mathbf{z}})=0\}
\end{equation*}
(understood as a real algebraic variety in $\mathbb{R}^2\times\mathbb{R}^{2n}$) is a union of strata. More precisely, we consider the following three types of strata:
\begin{enumerate}
\item[$\cdot$]
$A_I:=U\cap (V(f)\cap(\mathbb{C}\times \mathbb{C}^{*I}))$ 
for $I\in\mathscr{I}_{nv}(f_0)$;
\item[$\cdot$]
$B_I:=U\cap((\mathbb{C}\times \mathbb{C}^{*I})\setminus A_I)$ for $I\in\mathscr{I}_{nv}(f_0)$;
\item[$\cdot$]
$C_I:=U\cap(\mathbb{C}\times \mathbb{C}^{*I})$ 
for $I\in\mathscr{I}_v(f_0)$.
\end{enumerate} 
The (finite) collection 
\begin{equation*}
\mathscr{S}:=\{A_I,B_I\mid I\in\mathscr{I}_{nv}(f_0)\}\cup \{C_I\mid I\in\mathscr{I}_{v}(f_0)\}
\end{equation*}
is a real stratification of the set $U\cap(\mathbb{C}\times\mathbb{C}^n)$ for which $U\cap V(f)$ is a union of strata. Note that for $I=\emptyset$, which is an element of $\mathscr{I}_{v}(f_0)$, the stratum 
\begin{equation*}
C_{\emptyset}:=U\cap(\mathbb{C}\times \mathbb{C}^{*\emptyset})
\end{equation*} 
of $\mathscr{S}$ is nothing but the $t$-axis $U\cap(\mathbb{C}\times \{\mathbf{0}\})$.

\begin{definition}\label{def-admissible}
We say that the family of mixed polynomial functions $\{f_t\}$ is \emph{admissible} (at $t=0$) if it satisfies the above conditions ($\mbox{C}_1$) and ($\mbox{C}_2$), and if furthermore, there exists a positive number $\rho>0$ such that:
\begin{enumerate}
\item
$R\geq \rho$;
\item
$r_{nc}(f_t) \geq \rho$ for any sufficiently small $t$;
\end{enumerate}
where $R$ is given by Proposition \ref{lemmasmooth} and $r_{nc}(f_t)$ is the radius of local tameness of~$f_t$.
\end{definition}

In particular, if the family $\{f_t\}$ is admissible, then it is \emph{uniformly} locally tame along the vanishing coordinates subspaces---that is, $f_t$ is locally tame along $\mathbb{C}^I$ for any $I\in\mathscr{I}_{v}(f_0)$ and $r_{nc}(f_t)\geq\rho$ for all small $t$.

Here is our main result. 

\begin{theorem}\label{mt2}
If the family of mixed polynomial functions $\{f_t\}$ is admissible, then the canonical stratification $\mathscr{S}$ of $U\cap(\mathbb{C}\times\mathbb{C}^n)$ described above is Whitney $(b)$-regular. In particular, the corresponding family of mixed hypersurfaces $\{V(f_t)\}$ is Whitney equisingular, and hence, topologically equisingular. Moreover, the stratification $\mathscr{S}$ satisfies Thom's $a_f$ condition with respect to the complement of the set of critical points of $f$.
\end{theorem}
 
Theorem \ref{mt2} generalizes to mixed polynomial functions a similar result for holomorphic polynomial functions proved by the authors in \cite{EO}. However, as  mixed hypersurfaces do not carry any complex structure, the proof in the  ``mixed'' case requires essential new arguments. 
In particular, in the holomorphic case, the tangent space to $V(f)$ at a non-singular point $\mathbf{a}$ is the complex orthogonal complement of the  subspace of $\mathbb{C}\times \mathbb{C}^n$ generated by the gradient vector $\mbox{grad}f(\mathbf{a})$ of $f$ at $\mathbf{a}$. To show that the Whitney $(b)$-regularity and the Thom condition hold, we need to investigate the limit of $\mbox{grad}f(\mathbf{a})$ as $\mathbf{a}$ approaches some point located on a vanishing coordinates subspace for $f$. In the mixed case, the tangent space to $V(f)$ at a mixed non-singular point $\mathbf{a}$ is the real subspace of $\mathbb{C}\times \mathbb{C}^n$  whose vectors are orthogonal in $\mathbb{R}^2\times \mathbb{R}^{2n}$ to the ($\mathbb{R}$-linearly independent) vectors $\bar\partial g(\mathbf{a},\bar{\mathbf{a}})$ and $\bar\partial h(\mathbf{a},\bar{\mathbf{a}})$, where $g:=\Re(f)$ and $h:=\Im(f)$ are the real and imaginary parts of $f$ respectively. In this case, to show that the Whitney $(b)$-regularity and the Thom condition hold, we need to simultaneously investigate the limits of both $\bar\partial g(\mathbf{a},\bar{\mathbf{a}})$ and $\bar\partial h(\mathbf{a},\bar{\mathbf{a}})$.

Note that, in general, when $\mathbf{a}$  tends to some point located on a vanishing coordinates subspace for $f$, it may happen that the corresponding limits for the vectors $\bar\partial g(\mathbf{a},\bar{\mathbf{a}})$ and $\bar\partial h(\mathbf{a},\bar{\mathbf{a}})$ are no longer linearly independent over $\mathbb{R}$. This may causes serious trouble. However this problem can be easily resolved.

 By \cite[Theorem 20]{O1}, if the function $f$ itself is strongly non-degenerate and locally tame along its vanishing coordinates subspaces (i.e., the coordinates subspaces of $\mathbb{C}\times\mathbb{C}^n$ where $f$ vanishes), then $(\mathbb{C}\times\mathbb{C}^n,V(f))$ can also be endowed with a canonical stratification which satisfies the Thom $a_f$ condition. Note that in Theorem 20 of \cite{O1} the assumptions refer to the function $f$ itself whereas in Theorem \ref{mt2} above they refer to the members $f_t$ of the family.

Theorem \ref{mt2} is proved in Section \ref{pmt2}. The proof will show that in the special case where, in addition to the assumptions of the theorem, the functions $f_t$ are con\penalty 10000 venient, then the uniform local tameness assumption can be dropped (cf.~Remark \ref{rk-Ivide}). Thus, in this case, we have the following simplified statement.

\begin{theorem}\label{mt2-isolated}
Suppose that for all $t$ sufficiently small, the following two conditions are satisfied:
\begin{enumerate}
\item
the Newton boundary $\Gamma(f_t;\mathbf{z},\bar{\mathbf{z}})$ of $f_t$ at the origin is independent of $t$;
\item
the mixed polynomial function $f_t$ is convenient and strongly non-degenerate (in particular this implies that $f_t$ has an isolated mixed singularity at $\mathbf{0}$).
\end{enumerate}
Then the conclusions of Theorem \ref{mt2} still hold true.
\end{theorem}

\section{Proof of Theorem \ref{mt2}}\label{pmt2}

We first prove that $\mathscr{S}$ is Whitney $(b)$-regular. Then we show that it satisfies Thom's $a_f$ condition with respect to the complement of $\Sigma f$.

\subsection{Whitney's $(b)$-regularity}\label{powbrc}

First of all, observe that if $I\subseteq J$, then $\mathbb{C}^{*I}$ is contained in the closure $\overline{\mathbb{C}^{*J}}$ of $\mathbb{C}^{*J}$. Moreover, if  $I\subseteq J$ and $J\in\mathscr{I}_v(f_0)$, then  $I\in\mathscr{I}_v(f_0)$ too. Therefore, to prove the theorem, it suffices to check that the Whitney $(b)$-regularity condition holds for all pairs of strata satisfying one of the following three conditions:
\begin{enumerate}
\item 
$C_I \cap \overline{C_J}\not=\emptyset$ 
with $I\subseteq J$ and $I,J\in\mathscr{I}_{v}(f_0)$;
\item
$C_I \cap \overline{A_J}\not=\emptyset$ or $C_I \cap \overline{B_J}\not=\emptyset$
with $I\subseteq J$ and $I\in\mathscr{I}_v(f_0)$, $J\in\mathscr{I}_{nv}(f_0)$;
\item
$A_I \cap \overline{A_J}\not=\emptyset$,
$A_I \cap \overline{B_J}\not=\emptyset$ or
$B_I \cap \overline{B_J}\not=\emptyset$
with
$I\subseteq J$ and $I,J\in\mathscr{I}_{nv}(f_0)$.
\end{enumerate}
Except for the case of a pair of strata of the form $(A_J,C_I)$, the Whitney $(b)$-regularity condition immediately follows from Remark \ref{remarkstratification}. Thus, to prove our result, it suffices to show that for any $J\in\mathscr{I}_{nv}(f_0)$ and any $I\in\mathscr{I}_v(f_0)$, with $I\subseteq J$, the stratum
\begin{align*}
A_J:=U\cap(V(f)\cap (\mathbb{C}\times\mathbb{C}^{*J}))
\end{align*}
 is Whitney $(b)$-regular over the stratum
\begin{align*}
C_I:=U\cap(\mathbb{C}\times\mathbb{C}^{*I})
\end{align*}
 at any point $(\tau,\mathbf{q})=(\tau,q_1,\ldots,q_n)\in C_I \cap \overline{A_J}$ sufficiently close to the origin of $\mathbb{C}\times\mathbb{C}^n$. 
To simplify, we assume that $J=\{1,\ldots,n\}$ and $I=\{1,\ldots,m\}$ with $1\leq m\leq n-1$. (The argument is similar in the other cases; for $I=\emptyset$, see Remark \ref{rk-Ivide}.) In particular, $q_i\not=0$ if and only if $1\leq i\leq m$. 

Pick real analytic paths $(t(s),\mathbf{z}(s))$ and $(t'(s),{\mathbf{z}}'(s))$ 
such that:
\begin{enumerate}
\item
$(t(0),\mathbf{z}(0))=(t'(0),{\mathbf{z}}'(0))=(\tau,\mathbf{q})$;
\item
$(t'(s),{\mathbf{z}}'(s))\in C_I$ and $(t(s),\mathbf{z}(s))\in A_J$ for $s\not=0$.
\end{enumerate}
Write $\mathbf{z}(s):=(z_1(s),\ldots,z_n(s))$ and $\mathbf{z}'(s):=(z'_1(s),\ldots,z'_n(s))$, and look at the Taylor expansions:
\begin{align*}
& t(s)=\tau+b_0s+\cdots,\quad
z_i(s)=a_i s^{w_i}+b_i s^{w_i+1}+\cdots,\\
& t'(s)=\tau+b'_0 s+\cdots, \quad
z'_i(s)=q_i +b'_i s+\cdots,
\end{align*}
where $w_i=0$ and $a_i=q_i$ for $1\leq i\leq m$ while $w_i>0$ and $a_i\not=0$ for $i>m$. (As above, the dots stand for the higher order terms.) Note that, for any $s$, we have $z'_i(s)=0$ if $i>m$.
Set
\begin{equation*}
\ell(s):=(\ell_0(s),\ell_1(s),\ldots,\ell_n(s)),
\end{equation*}
where
\begin{equation*}
\left\{
\begin{aligned}
& \ell_0(s):=t(s)-t'(s)=(b_0-b'_0)s+\cdots,\\
& \ell_i(s):=z_i(s)-z'_i(s)=\left\{
\begin{aligned}
& (b_i-b'_i)s+\cdots && \mbox{ for}\quad 1\leq i\leq m,\\
& a_i s^{w_i}+\cdots && \mbox{ for}\quad m+1\leq i\leq n.
\end{aligned}
\right.
\end{aligned}
\right.
\end{equation*}
By reordering, we may suppose that
\begin{equation}\label{assumptionwi}
\begin{aligned}
w_{m+1}=\cdots=w_{m+m_1}<w_{m+m_1+1}=\cdots=w_{m+m_1+m_2}<\cdots\\
\cdots < w_{m+m_1+\cdots+m_{k-1}+1}=\cdots=w_{m+m_1+\cdots+m_{k}}=w_n,
\end{aligned}
\end{equation}
for some non-negative integers $m_1,\ldots,m_k\in\mathbb{N}$ with $m+m_1+\cdots+m_k=n$.
By Proposition \ref{prop-suitearc1}, to show that the pair of strata $(A_J,C_I)$ satisfies the Whitney $(b)$-regularity condition at the point $(\tau,\mathbf{q})$, it suffices to prove that 
\begin{equation}\label{lfwbrc}
\lim_{{s\to 0}\atop{s\not=0}} \frac{\ell(s)}{\Vert \ell(s) \Vert} \in 
\lim_{{s\to 0}\atop{s\not=0}} T_{(t(s),\mathbf{z}(s))} A_J,
\end{equation}
provided that these two limits exist. Hereafter, we suppose that these limits exist.

\begin{remark}
Observe that if $g:=\Re (f)$ and $h:=\Im (f)$ denote the real and imaginary parts of $f$ resp\-ectively, then, for all $s\not=0$,
\begin{equation}\label{rket}
T_{(t(s),\mathbf{z}(s))} A_J = \bar\partial g(\zeta(s))^\bot \cap 
\bar\partial h(\zeta(s))^\bot,
\end{equation}
where $\zeta(s):=(t(s),\bar t(s),\mathbf{z}(s),\bar{\mathbf{z}}(s))$.
(As usual, $\bar t(s)$ is the complex conjugate of $t(s)$ and $\bar{\mathbf{z}}(s):=(\bar{z}_1(s),\ldots,\bar{z}_n(s))$ where $\bar{z}_i(s)$ is the complex conjugate of $z_i(s)$. For any vector $v\in\mathbb{C}^N$, the symbol $v^\bot$ denotes the real orthogonal complement of the real subspace of $\mathbb{C}^N$ generated by $v$.)
\end{remark}

Let $o(\ell):=\mbox{min} \,
\{o(\ell_i) \mid 0\leq i\leq n\}$,
where $o(\ell_i)$ is the order in $s$ of the $i$th component $\ell_i(s)$ of $\ell(s)$.  Clearly,
$o(\ell)\leq w_{m+1}$, and
\begin{equation}\label{liml1}
\lim_{{s\to 0}\atop{s\not=0}}\frac{\ell(s)}{\vert s\vert^{o(\ell)}} = \left\{
\begin{aligned}
&(*,\underbrace{*,\ldots,*}_{m \mbox{ \tiny terms}},\underbrace{0,\ldots,0}_{n-m \mbox{ \tiny zeros}}) \quad\mbox{if}\quad o(\ell)<w_{m+1},\\
&(*,\underbrace{*,\ldots,*}_{m \mbox{ \tiny terms}},a_{m+1},\ldots,a_{m+m_1}, \underbrace{0,\ldots,0}_{n-m -m_1\mbox{ \tiny zeros}})
\quad\mbox{if}\quad o(\ell)=w_{m+1},
\end{aligned}
\right.
\end{equation}
where each term marked with a star ``$*$'' represents a complex number which may be zero or not.
Let $\mathbf{w}:=(w_1,\ldots,w_n)=(0,\ldots,0,w_{m+1},\ldots,w_n)$, and let 
\begin{equation*}
l_\mathbf{w}\colon \Gamma_{nc}(f_t;\mathbf{z},\bar{\mathbf{z}}) 
= \Gamma_{nc}(f_0;\mathbf{z},\bar{\mathbf{z}})\to \mathbb{R}
\end{equation*}
be the restriction of the linear map 
\begin{equation*}
(\xi_1,\ldots, \xi_n)\in \mathbb{R}^n\mapsto\sum_{1\leq i\leq n} \xi_i w_i\in \mathbb{R}. 
\end{equation*}
Denote by $d_\mathbf{w}$ the minimal value of $l_\mathbf{w}$, and write $\Delta_\mathbf{w}$ for the face of $\Gamma_{nc}(f_t;\mathbf{z},\bar{\mathbf{z}})=\Gamma_{nc}(f_0;\mathbf{z},\bar{\mathbf{z}})$ defined by the locus where $l_\mathbf{w}$ takes this minimal value. Clearly, $\Delta_\mathbf{w}$ is an essential non-compact face, and $I_{\Delta_\mathbf{w}}=I$.
Finally, let $\mathbf{a}:=(a_1,\ldots,a_n)$.
Then, for all $1\leq i\leq n$, 
\begin{equation}\label{f1-w}
\frac{\partial g}{\partial \bar{z}_i} (\zeta(s)) = \frac{\partial {(g_{\tau,\bar\tau})}_{\Delta_{\mathbf{w}}}}{\partial \bar{z}_i} (\mathbf{a},\bar{\mathbf{a}})\,  s^{d_\mathbf{w}-w_i}+\cdots, 
\end{equation}
 where $g_{\tau,\bar\tau}(\mathbf{z},\bar{\mathbf{z}}):=g(\tau,\bar\tau,\mathbf{z},\bar{\mathbf{z}})$ (i.e., $g_{\tau,\bar\tau}=\Re(f_\tau)$) and $(g_{\tau,\bar\tau})_{\Delta_{\mathbf{w}}}$ is the face function associated with $g_{\tau,\bar\tau}$ and $\Delta_{\mathbf{w}}$ (cf.~Remark \ref{rk-ffpri}), while
\begin{align}\label{ldt02-w}
\lim_{{s\to 0}\atop{s\not=0}}\biggl(\frac{1}{\vert s\vert^{d_{\mathbf{w}}-1}}\cdot
\frac{\partial g}{\partial \bar{t}} 
(\zeta(s))\biggr) 
= 0.
\end{align}
Similar expressions to (\ref{f1-w}) and (\ref{ldt02-w}) also hold true for the imaginary part $h$ of $f$. 
\begin{notation}\label{ltrunc}
For any vector $v=(v_1,\ldots,v_n)\in\mathbb{C}^n$ and any integers $i,j$ with $m+1\leq i\leq j\leq n$, we denote by $v_i^j$ the truncated vector $(v_i,\ldots,v_j)$.
\end{notation}
By the uniform local tameness (i.e., the condition $r_{nc}(f_t)\geq \rho$ for all small $t$), if $(\tau,\mathbf{q})$ is close enough to $(0,\mathbf{0})\in\mathbb{C}\times\mathbb{C}^n$, then the vectors
\begin{equation}\label{pp}
\begin{aligned}
& \bigl(\bar{\partial} {(g_{\tau,\bar\tau})}_{\Delta_{\mathbf{w}}}
(\mathbf{a},\bar{\mathbf{a}})\bigr)_{m+1}^{n}:=
\biggl(\frac{\partial {(g_{\tau,\bar\tau})}_{\Delta_{\mathbf{w}}}}{\partial \bar{z}_{m+1}} (\mathbf{a},\bar{\mathbf{a}}),\ldots,\frac{\partial {(g_{\tau,\bar\tau})}_{\Delta_{\mathbf{w}}}}{\partial \bar{z}_{n}} (\mathbf{a},\bar{\mathbf{a}})\biggr)\\
& \bigl(\bar{\partial} {(h_{\tau,\bar\tau})}_{\Delta_{\mathbf{w}}}
(\mathbf{a},\bar{\mathbf{a}})\bigr)_{m+1}^{n}:=
\biggl(\frac{\partial {(h_{\tau,\bar\tau})}_{\Delta_{\mathbf{w}}}}{\partial \bar{z}_{m+1}} (\mathbf{a},\bar{\mathbf{a}}),\ldots,\frac{\partial {(h_{\tau,\bar\tau})}_{\Delta_{\mathbf{w}}}}{\partial \bar{z}_{n}} (\mathbf{a},\bar{\mathbf{a}})\biggr)
\end{aligned}
\end{equation}
are linearly independent over $\mathbb{R}$.

\begin{notation}
Hereafter, to simplify, we set
\begin{align}\label{notation-vgvh}
v_{g,0}(s):=\frac{\partial g}{\partial \bar{t}}(\zeta(s))
\quad\mbox{and}\quad
v_{g,i}(s):=\frac{\partial g}{\partial \bar{z}_i}(\zeta(s))\quad (1\leq i\leq n),
\end{align}
and we write:\vskip 1mm
\begin{enumerate}
\item[$\cdot$]
$v_{g}(s):=(v_{g,1}(s),\ldots,v_{g,n}(s))$; \vskip 1mm
\item[$\cdot$]
$o(g):= \mbox{min}\, \{o(v_{g,i}) \mid 1\leq i\leq n\}$,
%\overset{\mbox{\tiny by (\ref{pp})}}{=}
%\mbox{min}\, \{o(v_{g,i}) \mid m+1\leq i\leq n\}$,\\
where $o(v_{g,i})$ is the order in $s$ of $v_{g,i}(s)$;\vskip 1mm
\item[$\cdot$]
$i(g):=\mbox{max}\, \{i\mid 1\leq i\leq n 
\mbox{ and }  o(v_{g,i})=o(g)\}$.
\end{enumerate}
\vskip 1mm
Similarly, we define $v_{h,0}(s)$, $v_{h,i}(s)$, $v_{h}(s)$, 
$o(h)$ and $i(h)$. 
\end{notation}

For simplicity, in the case where $i(g)=i(h)$, we assume $o(g)\leq o(h)$. (This is always possible by exchanging $g$ and $h$, considering $\sqrt{-1}\, f(t,\mathbf{z},\bar{\mathbf{z}})$ if necessary.)
The index $i(g)$ is called the \emph{essential} index of $v_{g}(s)$.  The coefficient of $s^{o(g)}$ in the expansion of $v_{g,i(g)}$ is called the \emph{leading} coefficient of $v_{g,i(g)}$. Similar definitions apply to the function~$h$.

Combined with the expression (\ref{f1-w}), the $\mathbb{R}$-linear independence of the vectors in (\ref{pp}) shows that $o(g)=\mbox{min}\, \{o(v_{g,i}) \mid m+1\leq i\leq n\}$, $i(g):=\mbox{max}\, \{i\mid m+1\leq i\leq n \mbox{ and } o(v_{g,i})=o(g)\}$, and
\begin{equation}\label{eqn-ogoh}
o(g)\leq d_{\mathbf{w}}-w_{m+1}\leq d_{\mathbf{w}}-1,
\end{equation}
with similar properties for $h$. Thus, by (\ref{ldt02-w}),
\begin{equation}\label{mr1}
\lim_{{s\to 0}\atop{s\not=0}}\biggl(\frac{1}{\vert s\vert^{o(g)}}\cdot v_{g,0} (s)\biggr) =
\lim_{{s\to 0}\atop{s\not=0}}\biggl(\frac{1}{\vert s\vert^{o(g)}}\cdot\frac{\partial g}{\partial \bar{t}} (\zeta(s))\biggr) = 0.
\end{equation}
Let us denote by $(\bar{\partial} g)_\infty$ the non-zero vector of $\mathbb{C}\times\mathbb{C}^n$ defined by
\begin{equation*}
(\bar{\partial} g)_\infty  :=\lim_{{s\to 0}\atop{s\not=0}} \frac{\bar{\partial} g(\zeta(s))}{\vert s\vert^{o(g)}} = \lim_{{s\to 0}\atop{s\not=0}} \frac{(v_{g,0} (s),v_g(s))}{\vert s\vert^{o(g)}}.
\end{equation*}
By (\ref{f1-w}), (\ref{eqn-ogoh}) and (\ref{mr1}), we have
\begin{equation}\label{mr}
(\bar{\partial} g)_\infty = \left\{
\begin{aligned}
& (0,\underbrace{0,\ldots,0}_{m+m_1 \mbox{ \tiny zeros}},\underbrace{*,\ldots,*}_{\ n-m-m_1 \mbox{ \tiny terms}}) \quad \mbox{if}\quad  o(g)<d_{\mathbf{w}}-w_{m+1},\\
& \biggl(0,\underbrace{0,\ldots,0}_{m\mbox{ \tiny zeros}},
\bigl(\bar{\partial} (g_{\tau,\bar\tau})_{\Delta_{\mathbf{w}}}(\mathbf{a},\bar{\mathbf{a}})\bigr)_{m+1}^{m+m_1},\underbrace{*,\ldots,*}_{n-m-m_1 \mbox{ \tiny terms}}\biggr)\\ 
& \mbox{\hskip 4.4cm if}\quad o(g)=d_{\mathbf{w}}-w_{m+1},
\end{aligned}
\right.
\end{equation}
where again each term marked with a star ``$*$'' represents a complex number which may be zero or not.
Define $(\bar{\partial} h)_\infty$ similarly. Then (\ref{mr}) also holds for $h$ instead of $g$.
Now the proof divides into two cases:
\begin{enumerate}
\item[($\mbox{A}_1$)]
either $i(g)\not=i(h)$ or $i(g)=i(h)$ and the leading coefficients of $v_{g,i(g)}$ and $v_{h,i(h)}$ are linearly independent over $\mathbb{R}$;
\item[($\mbox{A}_2$)]
$i(g)=i(h)$ and the leading coefficients of $v_{g,i(g)}$ and $v_{h,i(h)}$ are linearly dependent over $\mathbb{R}$.
\end{enumerate}

The first case is not so difficult. Indeed, under the assumption ($\mbox{A}_1$), the vectors $(\bar{\partial} g)_\infty$ and $(\bar{\partial} h)_\infty$ are linearly independent over $\mathbb{R}$. 
Therefore, 
\begin{align*}
\lim_{{s\to 0}\atop{s\not=0}} T_{(t(s),\mathbf{z}(s))} A_J 
& \overset{\mbox{\tiny by (\ref{rket})}}{=} 
\lim_{{s\to 0}\atop{s\not=0}} \bigl(\bar{\partial} g(\zeta(s))^{\bot} \cap \bar{\partial} h(\zeta(s))^{\bot}\bigr) \\
& =
\lim_{{s\to 0}\atop{s\not=0}} \biggl(\frac{\bar{\partial} g(\zeta(s))}{\vert s\vert^{o(g)}}^{\bot} \cap \frac{\bar{\partial} h(\zeta(s))}{\vert s\vert^{o(h)}}^{\bot}\biggr) 
 = (\bar{\partial} g)_\infty^{\bot} \cap (\bar{\partial} h)_\infty^{\bot}.
\end{align*}
Thus, to show that the relation (\ref{lfwbrc}) does hold true, we must prove that
\begin{equation*}
\ell_\infty:=\lim_{{s\to 0}\atop{s\not=0}} \frac{\ell(s)}{\Vert \ell(s) \Vert} \in (\bar{\partial} g)_\infty^{\bot} \cap (\bar{\partial} h)_\infty^{\bot}.
\end{equation*}
As $\Vert \ell(s) \Vert\sim c\vert s \vert^{o(\ell)}$ as $s\to 0$ (where $c$ is a constant), it follows immediately from (\ref{liml1}) and (\ref{mr}) that $\ell_\infty\in(\bar{\partial} g)_\infty^{\bot}$
 if $o(g)< d_{\mathbf{w}}-w_{m+1}$ or if $o(g)=d_{\mathbf{w}}-w_{m+1}$ and $o(\ell)<w_{m+1}$. To show this is also true when $o(g)=d_{\mathbf{w}}-w_{m+1}$ and $o(\ell)=w_{m+1}$, we must prove that
\begin{align}\label{pheazero-w}
\Re \biggl(\sum_{i=m+1}^{m+m_1}  \bar{a}_i\frac{\partial ({g_{\tau,\bar\tau}})_{\Delta_{\mathbf{w}}}}{\partial \bar{z}_{i}}(\mathbf{a},\bar{\mathbf{a}})\biggr)=0.
\end{align}
The polynomial $(g_{\tau,\bar\tau})_{\Delta_{\mathbf{w}}}$ is radially weighted homogeneous of type $(w_1,\ldots,w_n;d_{\mathbf{w}})$. Then, by the radial Euler identity, for any $\mathbf{z}=(z_1,\ldots,z_n)$, we have
\begin{align}\label{euler-radial}
d_{\mathbf{w}}\cdot ({g_{\tau,\bar\tau}})_{\Delta_{\mathbf{w}}}(\mathbf{z},\bar{\mathbf{z}})=\sum_{i=m+1}^n w_i z_i\frac{\partial ({g_{\tau,\bar\tau}})_{\Delta_{\mathbf{w}}}}{\partial z_{i}}(\mathbf{z},\bar{\mathbf{z}}) + w_i \bar{z}_i\frac{\partial ({g_{\tau,\bar\tau}})_{\Delta_{\mathbf{w}}}}{\partial \bar{z}_{i}}(\mathbf{z},\bar{\mathbf{z}}).
\end{align}
(Remind that $w_i=0$ for $1\leq i\leq m$.)
As $(t(s),\mathbf{z}(s))\in A_J$ for any $s\not=0$, the expression
\begin{equation*}
f(t(s),\mathbf{z}(s),\bar{\mathbf{z}}(s))=g(\zeta(s))+\sqrt{-1}\, h(\zeta(s))
\end{equation*} 
vanishes for all $s$, and hence $g(\zeta(s))=h(\zeta(s))=0$ for all $s$. It follows that $({g_{\tau,\bar\tau}})_{\Delta_{\mathbf{w}}}(\mathbf{a},\bar{\mathbf{a}})=0$. Therefore, applying (\ref{euler-radial}) with $\mathbf{z}=\mathbf{a}$ gives
\begin{equation}\label{euler-radial-ata}
\begin{aligned}
0 & = \sum_{i=m+1}^n w_i a_i\frac{\partial 
({g_{\tau,\bar\tau}})_{\Delta_{\mathbf{w}}}}{\partial z_{i}}(\mathbf{a},\bar{\mathbf{a}}) + w_i \bar{a}_i\frac{\partial ({g_{\tau,\bar\tau}})_{\Delta_{\mathbf{w}}}}{\partial \bar{z}_{i}}(\mathbf{a},\bar{\mathbf{a}})\\
& =2\Re \biggl(\sum_{i=m+1}^n w_i \bar{a}_i\frac{\partial ({g_{\tau,\bar\tau}})_{\Delta_{\mathbf{w}}}}{\partial \bar{z}_{i}}(\mathbf{a},\bar{\mathbf{a}})\biggr).
\end{aligned}
\end{equation}
Now, combined with (\ref{assumptionwi}) and (\ref{f1-w}), the equality $o(g)=d_{\mathbf{w}}-w_{m+1}$ implies 
\begin{equation*}
\frac{\partial {(g_{\tau,\bar\tau})}_{\Delta_{\mathbf{w}}}}{\partial \bar{z}_i} (\mathbf{a},\bar{\mathbf{a}})=0
\end{equation*}
for any $i>m+m_1$. Thus (\ref{euler-radial-ata}) says 
\begin{equation*}
0=2\Re \biggl(\sum_{i=m+1}^{m+m_1} w_i \bar{a}_i\frac{\partial ({g_{\tau,\bar\tau}})_{\Delta_{\mathbf{w}}}}{\partial \bar{z}_{i}}(\mathbf{a},\bar{\mathbf{a}})\biggr) =
2w_{m+1}\Re \biggl(\sum_{i=m+1}^{m+m_1}  \bar{a}_i\frac{\partial ({g_{\tau,\bar\tau}})_{\Delta_{\mathbf{w}}}}{\partial \bar{z}_{i}}(\mathbf{a},\bar{\mathbf{a}})\biggr).
\end{equation*}
As $w_{m+1}>0$, the equality (\ref{pheazero-w}) follows.

The same argument applied to the function $h$ shows that $\ell_\infty\in (\bar{\partial} h)_\infty^{\bot}$ as well, so that if ($\mbox{A}_1$) holds, then  $\ell_\infty\in (\bar{\partial} g)_\infty^{\bot}\cap (\bar{\partial} h)_\infty^{\bot}$ as desired. 

The case ($\mbox{A}_2$) is more delicate. Here, we shall use the same technique as that developed in the proof of Theorem 20 of \cite{O1}. First, we observe that if $c_g$ and $c_h$ denote the leading coefficients of $v_{g,i(g)}$ and $v_{h,i(h)}$, respectively, then ($\mbox{A}_2$) says that the quotient $c_h/c_g$ is a non-zero real number. Then we replace $v_h(s)$ by
\begin{equation*}
v_{h}'(s):=v_{h}(s)-\frac{c_h}{c_g}\,  s^{o(h)-o(g)} v_g(s)
\end{equation*}
in order to kill the coefficient $c_h$ of $s^{o(h)}$ in $v_{h,i(h)}$. 
Thus the $i$th component of $v_{h}'(s)$ is given by
\begin{equation}\label{fap}
\begin{aligned}
v_{h,i}'(s) & = v_{h,i}(s)-\frac{c_h}{c_g}\,  s^{o(h)-o(g)} v_{g,i}(s)\\
& = \biggl(\frac{\partial {(h_{\tau,\bar\tau})}_{\Delta_{\mathbf{w}}}}{\partial \bar{z}_{i}} (\mathbf{a},\bar{\mathbf{a}})-\frac{c_h}{c_g}\, \epsilon\frac{\partial {(g_{\tau,\bar\tau})}_{\Delta_{\mathbf{w}}}}{\partial \bar{z}_{i}} (\mathbf{a},\bar{\mathbf{a}})\biggr) s^{d_{\mathbf{w}}-w_i}+\cdots,
\end{aligned}
\end{equation}
where $\varepsilon$ is $0$ or $1$ according to $o(h)>o(g)$ or $o(g)=o(h)$.
By (\ref{pp}), the vectors
\begin{equation}\label{ppp}
\begin{aligned}
\bigl(\bar{\partial} {(g_{\tau,\bar\tau})}_{\Delta_{\mathbf{w}}}
(\mathbf{a},\bar{\mathbf{a}})\bigr)_{m+1}^{n}
\quad\mbox{and}\quad
\bigl(\bar{\partial} {(h_{\tau,\bar\tau})}_{\Delta_{\mathbf{w}}}
(\mathbf{a},\bar{\mathbf{a}})-\frac{c_h}{c_g}\, \epsilon\bar{\partial} {(g_{\tau,\bar\tau})}_{\Delta_{\mathbf{w}}}(\mathbf{a},\bar{\mathbf{a}})\bigr)_{m+1}^{n}
\end{aligned}
\end{equation}
are linearly independent over $\mathbb{R}$, and as above, this implies 
\begin{equation}\label{sap}
\begin{aligned}
o'(h) & :=\mbox{min}\{o(v'_{h,i})\mid 1\leq i\leq n\} = \mbox{min}\{o(v'_{h,i})\mid m+1\leq i\leq n\}\\
& \leq d_{\mathbf{w}}-w_{m+1}.
\end{aligned}
\end{equation}
(As usual, by $o(v'_{h,i})$ we mean the order in $s$ of $v'_{h,i}(s)$.)
Thus, by (\ref{ldt02-w}) (for $h$ instead of $g$), if $v'_{h,0}(s)$ denotes the expression $v_{h,0} (s)-({c_h}/{c_g})\, \epsilon\,  v_{g,0}(s)$, we have
\begin{equation}\label{tap}
\lim_{{s\to 0}\atop{s\not=0}}\biggl(\frac{1}{\vert s\vert^{o'(h)}}\cdot v'_{h,0}(s)\biggr)=0.
\end{equation}
Let $(\bar{\partial} h)'_{\infty}$ be the non-zero vector of $\mathbb{C}\times \mathbb{C}^{n}$ defined by
\begin{equation*}
(\bar{\partial} h)'_{\infty} := \lim_{{s\to 0}\atop{s\not=0}} \frac{(v'_{h,0} (s),v'_h(s))}{\vert s\vert^{o'(h)}}.
\end{equation*}
By (\ref{fap}), (\ref{sap}) and (\ref{tap}), we have
\begin{equation}\label{liml2}
(\bar{\partial} h)'_{\infty}=
\left\{
\begin{aligned}
& (0,\underbrace{0,\ldots,0}_{m+m_1 \mbox{ \tiny zeros}},\underbrace{*,\ldots,*}_{\ n-m-m_1 \mbox{ \tiny terms}}) \quad \mbox{if}\quad  o'(h)<d_{\mathbf{w}}-w_{m+1},\\
& \biggl(0,\underbrace{0,\ldots,0}_{m\mbox{ \tiny zeros}},
\bigl(\bar{\partial} (h_{\tau,\bar\tau})_{\Delta_{\mathbf{w}}}(\mathbf{a},\bar{\mathbf{a}})
-\frac{c_h}{c_g}\, \varepsilon\,  \bar{\partial} (g_{\tau,\bar\tau})_{\Delta_{\mathbf{w}}}(\mathbf{a},\bar{\mathbf{a}})
\bigr)_{m+1}^{m+m_1},\underbrace{*,\ldots,*}_{n-m-m_1 \mbox{ \tiny terms}}\biggr)\\ 
& \mbox{\hskip 4.4cm if}\quad o'(h)=d_{\mathbf{w}}-w_{m+1}.
\end{aligned}
\right.
\end{equation}
Here, as above, the proof divides into two cases:
\begin{enumerate}
\item[($\mbox{A}'_1$)]
either $i'(h)\not=i(g)$ or $i'(h)=i(g)$ and the leading coefficients of $v'_{h,i'(h)}$  and $v_{g,i(g)}$ are linearly independent over $\mathbb{R}$;
\item[($\mbox{A}'_2$)]
$i'(h)=i(g)$ and the leading coefficients of  $v'_{h,i'(h)}$ and $v_{g,i(g)}$ are linearly dependent over $\mathbb{R}$;
\end{enumerate}
where $i'(h)\geq m+1$ is the essential index of $v'(h)$.

Under the assumption ($\mbox{A}'_1$), the vectors $(\bar{\partial} g)_\infty$ et $(\bar{\partial} h)'_\infty$ are linearly independent over $\mathbb{R}$. Therefore,
\begin{align*}
\lim_{{s\to 0}\atop{s\not=0}} T_{(t(s),\mathbf{z}(s))} A_J & =
\lim_{{s\to 0}\atop{s\not=0}} \bigl(\bar{\partial} g(\zeta(s))^{\bot} \cap \bar{\partial} h(\zeta(s))^{\bot}\bigr) 
 = \lim_{{s\to 0}\atop{s\not=0}} \bigl(\bar{\partial} g(\zeta(s))^{\bot} \cap (v'_{h,0}{(s), v'_h(s))}^{\bot}\bigr) \\
& = \lim_{{s\to 0}\atop{s\not=0}} \biggl(\frac{\bar{\partial} g(\zeta(s))}{\vert s\vert^{o(g)}}^{\bot} \cap \frac{\bigl(v'_{h,0}(s), v'_h(s)\bigr)}{\vert s\vert^{o'(h)}}^{\bot}\biggr) 
 = (\bar{\partial} g)_\infty^{\bot} \cap {(\bar{\partial} h)'_\infty}^{\bot},
\end{align*}
and in order to show that the relation (\ref{lfwbrc}) does hold true, we must prove that
\begin{equation*}
\ell_\infty\in (\bar{\partial} g)_\infty^{\bot} \cap 
{(\bar{\partial} h)'_\infty}^{\bot}.
\end{equation*}
We already know that $\ell_\infty \in (\bar{\partial} g)_\infty^{\bot}$. It remains to prove that $\ell_\infty \in {(\bar{\partial} h)'_\infty}^{\bot}$.
If $o'(h)< d_{\mathbf{w}}-w_{m+1}$ or $o'(h)=d_{\mathbf{w}}-w_{m+1}$ and $o(\ell)<w_{m+1}$, this follows immediately from the relations (\ref{liml1}) and (\ref{liml2}). If $o'(h)=d_{\mathbf{w}}-w_{m+1}$ and $o(\ell)=w_{m+1}$, then we must prove that
\begin{align}\label{egaliteam-w}
\Re\biggl(\sum_{i=m+1}^{m+m_1} \bar{a}_i\biggl(\frac{\partial ({h_{\tau,\bar\tau}})_{\Delta_{\mathbf{w}}}}{\partial \bar{z}_{i}}(\mathbf{a},\bar{\mathbf{a}})-\frac{c_h}{c_g}\, \varepsilon \frac{\partial ({g_{\tau,\bar\tau}})_{\Delta_{\mathbf{w}}}}{\partial \bar{z}_{i}}(\mathbf{a},\bar{\mathbf{a}})\biggr)\biggr)=0.
\end{align}
As $(t(s),\mathbf{z}(s))\in A_J$ for any $s\not=0$, we have $g(\zeta(s))=h(\zeta(s))=0$ for all $s$.
Thus,
\begin{align*}
\frac{d}{ds}\biggl(h(\zeta(s))-\frac{c_h}{c_g}\, \varepsilon\,  s^{o(h)-o(g)} g(\zeta(s))\biggr)=0,
\end{align*}
and by taking the coefficient of the term with lowest degree (i.e., the coefficient of $s^{d_{\mathbf{w}}-1}$), we get
\begin{align}\label{numersuppl}
2 \Re\biggl(\sum_{i=m+1}^n w_i \bar{a}_i\biggl(\frac{\partial ({h_{\tau,\bar\tau}})_{\Delta_{\mathbf{w}}}}{\partial \bar{z}_{i}}(\mathbf{a},\bar{\mathbf{a}})-\frac{c_h}{c_g}\, \varepsilon \frac{\partial ({g_{\tau,\bar\tau}})_{\Delta_{\mathbf{w}}}}{\partial \bar{z}_{i}}(\mathbf{a},\bar{\mathbf{a}})\biggr)\biggr)=0.
\end{align}
As $o'(h)=d_{\mathbf{w}}-w_{m+1}$, we have
\begin{align*}
\frac{\partial ({h_{\tau,\bar\tau}})_{\Delta_{\mathbf{w}}}}{\partial \bar{z}_{i}}(\mathbf{a},\bar{\mathbf{a}})-\frac{c_h}{c_g}\, \varepsilon \frac{\partial ({g_{\tau,\bar\tau}})_{\Delta_{\mathbf{w}}}}{\partial \bar{z}_{i}}(\mathbf{a},\bar{\mathbf{a}})=0
\end{align*}
for all $m+m_1< i\leq n$. Thus (\ref{numersuppl}) says 
\begin{align*}
2 w_{m+1}\Re\biggl(\sum_{i=m+1}^{m+m_1} \bar{a}_i\biggl(\frac{\partial ({h_{\tau,\bar\tau}})_{\Delta_{\mathbf{w}}}}{\partial \bar{z}_{i}}(\mathbf{a},\bar{\mathbf{a}})-\frac{c_h}{c_g}\, \varepsilon \frac{\partial ({g_{\tau,\bar\tau}})_{\Delta_{\mathbf{w}}}}{\partial \bar{z}_{i}}(\mathbf{a},\bar{\mathbf{a}})\biggr)\biggr)=0.
\end{align*}
As $w_{m+1}>0$, the equality (\ref{egaliteam-w}) follows, and this completes the proof under the assumption ($\mbox{A}'_1$).

Now if the essential indices $i'(h)$ and $i(g)$ of $v'(h)$ and $v(g)$, respectively, are equal and if the leading coefficients of $v'_{h,i'(h)}$ and $v_{g,i(g)}$  are linearly dependent over $\mathbb{R}$ (i.e., if we are under the assumption ($\mbox{A}'_2$)), then we replace $v_{h}'(s)$ by
\begin{equation*}
v_{h}''(s):=v'_{h}(s)-\frac{c'_h}{c_g}\,  s^{o'(h)-o(g)} v_g(s)
\end{equation*}
in order to kill the coefficient $c'_h$ of $s^{o'(h)}$ in $v'_{h,i'(h)}$, and we repeat the above argument according to the cases ($\mbox{A}''_1$) and ($\mbox{A}''_2$) described below:
\begin{enumerate}
\item[($\mbox{A}''_1$)]
either $i''(h)\not=i(g)$ or $i''(h)=i(g)$ and the leading coefficients of $v''_{h,i''(h)}$ and $v_{g,i(g)}$  are linearly independent over $\mathbb{R}$;
\item[($\mbox{A}''_2$)]
$i''(h)=i(g)$ and the leading coefficients of $v''_{h,i''(h)}$ and $v_{g,i(g)}$  are linearly dependent over $\mathbb{R}$;
\end{enumerate}
where $i''(h)$ is the essential index of $v_{h}''(s)$.
We continue this process until we get a vector $v^{(k)}_{h}(s)$ whose essential index $i^{(k)}(h)$ is such that:
\begin{enumerate}
\item[($\mbox{A}^{(k)}_1$)]
either $i^{(k)}(h)\not=i(g)$ or $i^{(k)}(h)=i(g)$ but the leading coefficients of $v^{(k)}_{h,i^{(k)}(h)}$ and $v_{g,i(g)}$  are linearly independent over $\mathbb{R}$. 
\end{enumerate}
Then we can conclude just as in the case ($A_1'$). 
To see that this process terminates after a finite number $k$ of steps, observe that at each step (i.e., $v'(h),v''(h),\ldots$), the vectors ``corresponding'' to (4.16) are always linearly independent over $\mathbb{R}$. Therefore all the orders $o'(h),o''(h),\ldots$ remain less than or equal to $d_{\mathbf{w}}-w_{m+1}$, and since the sequence $o'(h),o''(h),\ldots$ is monotone increasing, the process must stop after a finite number $k$ of steps.

This completes the proof of the part of Theorem \ref{mt2} concerning the Whitney $(b)$-regularity condition. It remains to prove the part about the Thom condition. This is done in the next section.

\subsection{Thom's $a_f$ condition}\label{prooftc}
Although the argument is simpler, the idea to prove the Thom $a_f$ condition is  essentially the same as that used to obtain the Whitney $(b)$-regularity. 
First, by Proposition \ref{prop-suitearc2}, we observe that it suffices to show that for any $I\subseteq\{1,\ldots,n\}$, any $(\tau,\mathbf{q}):=(\tau,q_1,\ldots,q_n) \in V(f)\cap (\mathbb{C}\times \mathbb{C}^{*I})$ close enough to the origin, and any real analytic path 
\begin{align*}
(t(s),\mathbf{z}(s)):=(t(s),z_1(s),\ldots,z_n(s)) 
\end{align*}
with $(t(0),\mathbf{z}(0))=(\tau,\mathbf{q})$ and $(t(s),\mathbf{z}(s))\notin \Sigma f$ for $s\not=0$, the following condition on the tangent spaces holds true:
\begin{align}\label{lcamplct}
\lim_{{s\to 0}\atop{s\not=0}} T_{(t(s),\mathbf{z}(s))} V(f-f(t(s),\mathbf{z}(s),\bar{\mathbf{z}}(s))) \supseteq T_{(\tau,\mathbf{q})} (V(f)\cap (\mathbb{C}\times \mathbb{C}^{*I})),
\end{align}
provided that this limit exists. Hereafter we assume that it exists.
By the strong non-degeneracy condition, we may assume that $I\in\mathscr{I}_{v}(f_0)$, so that 
\begin{equation*}
U\cap (V(f)\cap (\mathbb{C}\times \mathbb{C}^{*I}))=U\cap (\mathbb{C}\times \mathbb{C}^{*I})=C_I.
\end{equation*} 
(Indeed, if $I\in\mathscr{I}_{nv}(f_0)$, then $(\tau,\mathbf{q})$ is a mixed non-singular point of the mixed hypersurface $V(f)$, and in this case the relation (\ref{lcamplct}) is obviously satisfied.) Again, to simplify, we assume that $I=\{1,\ldots,m\}$ with $1\leq m\leq n-1$. (For $I=\emptyset$, see Remark \ref{rk-Ivide}.) 
We keep the same notation as in the previous section. In particular,
we set
\begin{align*}
t(s)=\tau+b_0s+\cdots
\quad\mbox{and}\quad
z_i(s)=a_i s^{w_i}+b_i s^{w_i+1}+\cdots,
\end{align*}
where $w_i=0$ and $a_i=q_i$ for $1\leq i\leq m$ while $w_i>0$ and $a_i\not=0$ for $i>m$. We also suppose that (\ref{assumptionwi}) holds true. 
Of course, for $1\leq i\leq n$, the expression (\ref{f1-w}) is still valid---that is,
\begin{equation}\label{f1-t}
\frac{\partial g}{\partial \bar{z}_i} (\zeta(s)) = \frac{\partial {(g_{\tau,\bar\tau})}_{\Delta_{\mathbf{w}}}}{\partial \bar{z}_i} (\mathbf{a},\bar{\mathbf{a}})\, s^{d_\mathbf{w}-w_i}+\cdots,
\end{equation}
with a similar expression for $h$.
By the uniform local tameness, if $(\tau,\mathbf{q})$ is close enough to the origin $(0,\mathbf{0})\in\mathbb{C}\times\mathbb{C}^n$, then the vectors in (\ref{pp}) are still linearly independent over $\mathbb{R}$ too. Then, again, we divide the proof into the cases ($\mbox{A}_1$) and ($\mbox{A}_2$) described in  \S \ref{powbrc}.

Under the assumption ($\mbox{A}_1$), the vectors $(\bar\partial g)_\infty$ and $(\bar\partial h)_\infty$, which are contained in $\{0\}\times\{\mathbf{0}\}\times \mathbb{C}^{n-m} \subseteq \mathbb{C}\times \mathbb{C}^{m}\times \mathbb{C}^{n-m}$, are linearly independent over $\mathbb{R}$. Therefore,
\begin{align*}
 \lim_{{s\to 0}\atop{s\not=0}} T_{(t(s),\mathbf{z}(s))} V(f-f(t(s),\mathbf{z}(s),\bar{\mathbf{z}}(s))) 
= (\bar{\partial} g)_\infty^{\bot} \cap (\bar{\partial} h)_\infty^{\bot}.
\end{align*}
As $(\bar{\partial} g)_\infty^{\bot} \cap (\bar{\partial} h)_\infty^{\bot}$ contains $\mathbb{C}\times \mathbb{C}^{m}\times \{\mathbf{0}\}=\mathbb{C}\times \mathbb{C}^I$, the relation (\ref{lcamplct}) holds true in this case.

Under the assumption ($\mbox{A}_2$), as above, we replace $v_h(s)$ by $v_{h}'(s)$, and we consider the two possibilities ($\mbox{A}'_1$) and ($\mbox{A}'_2$) described in \S \ref{powbrc}. In the first case, the vectors $(\bar{\partial} g)_\infty$ et $(\bar{\partial} h)'_\infty$ (which are contained in $\{0\}\times \{\mathbf{0}\}\times \mathbb{C}^{n-m}$) are linearly independent over $\mathbb{R}$, and hence
\begin{align*}
 \lim_{{s\to 0}\atop{s\not=0}} T_{(t(s),\mathbf{z}(s))} V(f-f(t(s),\mathbf{z}(s),\bar{\mathbf{z}}(s)))  
= (\bar{\partial} g)_\infty^{\bot} \cap {(\bar{\partial} h)'_\infty}^{\bot}.
\end{align*}
Again as $(\bar{\partial} g)_\infty^{\bot} \cap {(\bar{\partial} h)'_\infty}^{\bot} \supseteq \mathbb{C}\times \mathbb{C}^{m}\times \{\mathbf{0}\}=\mathbb{C}\times \mathbb{C}^I$, the relation (\ref{lcamplct}) holds true.
In the second case, as in \S \ref{powbrc}, we repeat the argument until we get a vector $v^{(k)}_{h}(s)$ whose essential index $i^{(k)}(h)$ satisfies the condition ($\mbox{A}^{(k)}_1$).
Then we conclude by an argument similar to that used under the assumption ($\mbox{A}_1'$).

This completes the proof of Theorem \ref{mt2}.

\begin{remark}\label{rk-Ivide}
In the above proof, we have assumed $I\not=\emptyset$ (i.e., $m\geq 1$). However, a straightforward modification shows that the argument still works when $I=\emptyset$. Simply, observe that if $I=\emptyset$, then the face $\Delta_{\mathbf{w}}$ is compact, and hence, instead of the uniform local tameness, we must invoke the strong non-degeneracy. In particular, we obtain Theorem \ref{mt2-isolated}.
\end{remark}

\section{Examples of admissible families}

In this section, we give some examples of admissible families of mixed polynomial functions with non-isolated mixed singularities. Thus, by Theorem \ref{mt2}, all these families are Whitney equisingular and satisfy the Thom condition.

\subsection{A family of mixed curves with non-isolated mixed singularities}\label{example51} 
Consider the family given by the mixed polynomial function 
\begin{equation*}
f(t,z_1,z_2,\bar{z}_1,\bar{z}_2):=\bar{z}_1^2z_2^3+z_1^3\bar{z}_2^2+tz_1^2z_2^4.
\end{equation*}
Let $A=(2,3)$ and $B=(3,2)$. Clearly, for all small $t$, the mixed singular locus of $V(f_t)$ (i.e., the set of points of $V(f_t)$ which are critical points of the mixed polynomial function $f_t$) in a~suf\-ficiently small open neighbourhood of the origin in $\mathbb{C}^2$ contains the coordinates axes.
The non-compact Newton boundary $\Gamma_{nc}(f_t;\mathbf{z},\bar{\mathbf{z}})$, which is clearly independent of $t$, has two (compact) $0$-dimensional faces (the points $A$ and $B$), one compact $1$-dimensional face (the segment $\overline{AB}$), and two essential non-compact faces: $\Xi_1:=B+\mathbb{R}_+\mathbf{e}_1$ and $\Xi_2:=A+\mathbb{R}_+\mathbf{e}_2$.  See Figure \ref{fig2}, left-hand side. We easily check that for each compact face $\Delta$ and each $t$, the face function $(f_t)_{\Delta}$ has no critical point on $(\mathbb{C}^*)^2$ (i.e., $f_t$ is strongly non-degenerate).
We claim that for any $I\in\mathscr{I}_v(f_0)=\mathscr{I}_v(f_t)$, the family $\{f_t\}$ is uniformly locally tame along $\mathbb{C}^I$. Indeed, a trivial calculation shows that for any fixed $u_1\in\mathbb{C}^*$, the mixed polynomial function 
\begin{equation*}
z_2\mapsto(f_t)_{\Xi_1}(u_1,z_2,\bar{u}_1,\bar{z}_2):=u_1^3\bar{z}_2^2
\end{equation*}
of the variable $z_2$ has no critical point on $\mathbb{C}^{*\{1,2\}}_{u_1}$.
Similarly, for any fixed $u_2\in\mathbb{C}^*$ with $\vert u_2\vert<1/\vert t\vert$ (if $t\not=0$), the mixed polynomial function 
\begin{equation*}
z_1\mapsto(f_t)_{\Xi_2}(z_1,u_2,\bar{z}_1,\bar{u}_2):=\bar{z}_1^2u_2^3+tz_1^2u_2^4
\end{equation*}
of the variable $z_1$ has no critical point on $\mathbb{C}^{*\{1,2\}}_{u_2}$. So we can take
\begin{equation*}
r_{nc}(f_t)=\left\{
\begin{aligned}
1/\vert t\vert\mbox{ for } t\not=0,\\
\infty \mbox{ for } t=0,
\end{aligned}
\right.
\end{equation*}
and we have $r_{nc}(f_t)>\rho:=1$ for all $t$ with $\vert t\vert<1$. It follows that
the family $\{f_t\}$ is admissible. 

\begin{figure}[t]
\includegraphics[scale=1.8]{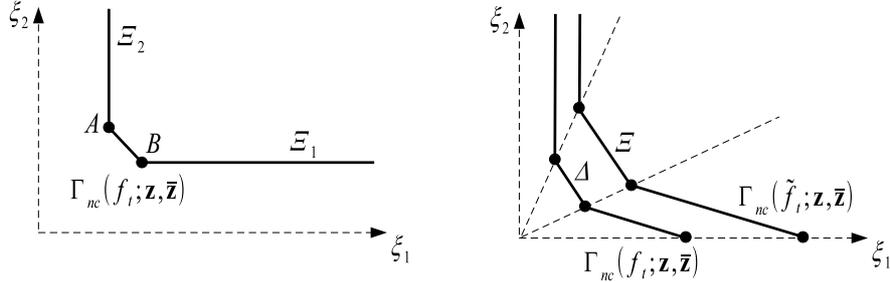}
\caption{Example \ref{example51} (left) and Example \ref{example53} (right)}
\label{fig2}
\end{figure}

\subsection{Families with ``big'' exponents for $t$-dependent monomials}
Let $\ell(\mathbf{z},\bar{\mathbf{z}})$ be a mixed polynomial function on $\mathbb{C}^n$ and let $k(t,\mathbf{z},\bar{\mathbf{z}})$ be a mixed polynomial function on $\mathbb{C}\times \mathbb{C}^n$. As usual, we write $k_t(\mathbf{z},\bar{\mathbf{z}}):=k(t,\mathbf{z},\bar{\mathbf{z}})$. Suppose that for all small $t$, 
\begin{equation*}
\Gamma_+(k_t;\mathbf{z},\bar{\mathbf{z}})\subseteq\Gamma_+(\ell;\mathbf{z},\bar{\mathbf{z}})
\quad\mbox{and}\quad
\Gamma_{nc}(k_t;\mathbf{z},\bar{\mathbf{z}})\cap\Gamma_{nc}(\ell;\mathbf{z},\bar{\mathbf{z}})=\emptyset. 
\end{equation*} 
Under this assumption, if $\ell$ is strongly non-degenerate and locally tame along its vanishing coordinates subspaces, then the family $\{f_t\}$ defined by 
\begin{equation*}
f_t(\mathbf{z},\bar{\mathbf{z}}):=\ell(\mathbf{z},\bar{\mathbf{z}})+k_t(\mathbf{z},\bar{\mathbf{z}})
\end{equation*} 
is admissible. For example, the family given by 
\begin{equation*}
f(t,z_1,z_2,\bar{z}_1,\bar{z}_2) :=
\bar{z}_1^2z_2^3+z_1^3\bar{z}_2^2+(1+t)z_1^3\bar{z}_2^3
\end{equation*} 
is admissible.

\subsection{Admissible families and mixed branched coverings}\label{example53} 
Take a positive integer $\delta\in\mathbb{N}^*$ and choose non-negative integer vectors $\mathbf{\nu}=(\nu_1,\ldots,\nu_n)$ and $\mathbf{\mu}=(\mu_1,\ldots,\mu_n)$ such that for all $1\leq i\leq n$:
\begin{enumerate}
\item
$\nu_i\in\mathbb{N}^*$ and $\mu_i\in\mathbb{N}$;
\item
$\nu_i>\mu_i$;
\item
 $\nu_i+\mu_i=\delta$;
\end{enumerate}
Then consider the mixed branched covering 
$\varphi\colon \mathbb{C}^n\to\mathbb{C}^n$
de\penalty 10000 fined~by 
\begin{equation*}
\mathbf{z}=(z_1,\ldots, z_n)\mapsto \varphi(\mathbf{z},\bar{\mathbf{z}}):=(z_1^{\nu_1}\bar{z}_1^{\mu_1}, \ldots, z_n^{\nu_n}\bar{z}_n^{\mu_n})
\end{equation*}
and whose ramification locus is given by the coordinates hyperplanes $z_i=0$ ($1\leq i\leq n$). In \cite[Proposition 22]{O1}, the second author showed that if $\ell(\mathbf{z})$ is a (strongly) non-degenerate \emph{holomorphic} polynomial function which is locally tame along $\mathbb{C}^I$ for any $I\in\mathscr{I}_v(\ell)$, then the mixed polynomial function
\begin{equation*}
\tilde \ell(\mathbf{z},\bar{\mathbf{z}}):=\ell\circ \varphi(\mathbf{z},\bar{\mathbf{z}})=\ell(z_1^{\nu_1}\bar{z}_1^{\mu_1}, \ldots, z_n^{\nu_n}\bar{z}_n^{\mu_n})
\end{equation*}
is strongly non-degenerate, 
$\mathscr{I}_v(\tilde \ell)=\mathscr{I}_v(\ell)$, and $\tilde \ell$ is locally tame along $\mathbb{C}^I$ for any $I\in\mathscr{I}_v(\tilde \ell)$. Actually, the proof shows that if $\{f_t\}$ is an admissible family of \emph{holomorphic} polynomial functions, then the family of mixed polynomial functions $\{\tilde f_t\}$ defined by 
\begin{equation*}
\tilde f_t(\mathbf{z},\bar{\mathbf{z}}) :=
f_t\circ \varphi(\mathbf{z},\bar{\mathbf{z}})
\end{equation*}
is admissible too.
Indeed, it is not difficult to see that the independence of $\Gamma_{nc}(f_t;\mathbf{z})$ with respect to $t$ implies that of $\Gamma_{nc}(\tilde f_t;\mathbf{z},\bar{\mathbf{z}})$. Also, it is easy to check that $\mathscr{I}_v(\tilde f_t)=\mathscr{I}_v(f_t)$ and $\mathscr{I}_{nv}(\tilde f_t)=\mathscr{I}_{nv}(f_t)$. To see that $\tilde f_t$ is strongly non-degenerate, we argue by contradiction. Take any compact face $\Xi\subseteq \Gamma(\tilde f_t;\mathbf{z},\bar{\mathbf{z}})$, and suppose that the face function $(\tilde f_t)_\Xi$ has a critical point $\mathbf{a}:=(a_1,\ldots,a_n)$ in $(\mathbb{C}^*)^n$, that is, there exists $\lambda\in\mathbb{S}^1$ such that  
\begin{equation}\label{degface}
\overline{\partial (\tilde f_t)_\Xi} (\mathbf{a},\bar{\mathbf{a}}) = 
\lambda \bar\partial (\tilde f_t)_\Xi (\mathbf{a},\bar{\mathbf{a}}).
\end{equation}
Clearly, 
\begin{equation*}
(\tilde f_t)_\Xi(\mathbf{z},\bar{\mathbf{z}})=(f_t)_\Delta\circ\varphi(\mathbf{z},\bar{\mathbf{z}})=(f_t)_\Delta(z_1^{\nu_1}\bar{z}_1^{\mu_1}, \ldots, z_n^{\nu_n}\bar{z}_n^{\mu_n}), 
\end{equation*}
where $\Delta$ is the compact face of $\Gamma(f_t;\mathbf{z})$ corresponding to $\Xi$, that is, if $\Delta\cap\mathbb{N}^n=\{(\alpha_1,\ldots,\alpha_n),(\beta_1,\ldots,\beta_n),\ldots\}$, then $\Xi\cap\mathbb{N}^n=\{\delta(\alpha_1,\ldots,\alpha_n),\delta(\beta_1,\ldots,\beta_n),\ldots\}$. See Figure \ref{fig2}, right-hand side. 
Therefore, by (\ref{degface}), for all $1\leq i\leq n$,
\begin{equation*}
\bar{a}_i^{\mu_i-1}a_i^{\mu_i} \biggl( \nu_i\bar{a}_i^{\nu_i-\mu_i} \overline{\frac{\partial (f_t)_\Delta}{\partial z_i}}(\varphi(\mathbf{a},\bar{\mathbf{a}})) - \lambda \mu_i a_i^{\nu_i-\mu_i} \frac{\partial (f_t)_\Delta}{\partial z_i}(\varphi(\mathbf{a},\bar{\mathbf{a}}))\biggr)=0.
\end{equation*}
As $a_i\not=0$, this implies
\begin{equation}\label{relex53}
\nu_i\, \vert\bar{a}_i\vert^{\nu_i-\mu_i}\, \bigg\vert\overline{\frac{\partial (f_t)_\Delta}{\partial z_i}}(\varphi(\mathbf{a},\bar{\mathbf{a}}))\bigg\vert = \mu_i \, \vert\lambda\vert\,  \vert a_i\vert^{\nu_i-\mu_i}\,  \bigg\vert\frac{\partial (f_t)_\Delta}{\partial z_i}(\varphi(\mathbf{a},\bar{\mathbf{a}}))\bigg\vert.
\end{equation}
As $f_t$ is (strongly) non-degenerate,
\begin{equation*}
\frac{\partial (f_t)_\Delta}{\partial z_i}(\varphi(\mathbf{a},\bar{\mathbf{a}})) \not=0,
\end{equation*}
and therefore (\ref{relex53}) implies $\nu_i=\mu_i$, which contradicts the above assumption (2). 

\begin{claim} 
The family $\{\tilde f_t\}$ is uniformly locally tame along the vanishing coordinates subspaces.
\end{claim}

\begin{proof}
By hypothesis, we know that the family of holomorphic polynomial functions $\{f_t\}$ is uniformly locally tame along the vanishing coordinates subspaces (i.e., $f_t$ is locally tame along $\mathbb{C}^I$ for any $I\in\mathscr{I}_v(f_t)=\mathscr{I}_v(f_0)$ and there exists $\rho>0$ such that $r_{nc}(f_t)\geq \rho$ for all small $t$). Without loss of generality, we may assume that $\rho<1$. Consider a subset $I\in\mathscr{I}_v(\tilde f_t)=\mathscr{I}_v(\tilde f_0)$. For simplicity, let us assume that $I=\{1,\ldots,m\}$. Let $u_1,\ldots,u_m$ be non-zero complex numbers such that 
\begin{equation}\label{inegdui}
\vert u_1 \vert^2 + \cdots + \vert u_m \vert^2\leq \rho^2.
\end{equation}
Take any essential non-compact face $\Xi\in\Gamma_{nc}(\tilde f_t;\mathbf{z},\bar{\mathbf{z}})$ such that $I_\Xi=I$, and consider the corresponding face $\Delta\in\Gamma_{nc}(f_t;\mathbf{z})$. (Note that $I_\Delta=I$ too.) We want to show that the face function $(\tilde f_t)_\Xi$ has no critical point on $\mathbb{C}^{*\{1,\ldots,n\}}_{u_1,\ldots,u_m}$ as a mixed polynomial function of the variables $z_{m+1},\ldots, z_n$. Again, we argue by contradiction. Suppose $(u_1,\ldots,u_m,a_{m+1},\ldots,a_n)$ is a critical point. Then, reminding Notation \ref{ltrunc}, there exists $\lambda\in\mathbb{S}^1$ such that for all $m+1\leq i\leq n$,
\begin{equation}\label{getcontradiction}
\begin{aligned}
\bar{a}_i^{\mu_i-1}a_i^{\mu_i} \biggl( \nu_i\bar{a}_i^{\nu_i-\mu_i} 
& \overline{\frac{\partial (f_t)_\Delta}{\partial z_i}}(\varphi(\mathbf{u}_1^m,\mathbf{a}_{m+1}^n,\bar{\mathbf{u}}_1^m,\bar{\mathbf{a}}_{m+1}^n)) - \\
& \lambda \mu_i a_i^{\nu_i-\mu_i} \frac{\partial (f_t)_\Delta}{\partial z_i}(\varphi(\mathbf{u}_1^m,\mathbf{a}_{m+1}^n,\bar{\mathbf{u}}_1^m,\bar{\mathbf{a}}_{m+1}^n))\biggr)=0.
\end{aligned}
\end{equation}
As $\rho<1$, the inequality (\ref{inegdui}) implies
\begin{equation*}
\vert u_1^{\nu_1}\bar{u}_1^{\mu_1}\vert^2 + \ldots + 
\vert u_m^{\nu_m}\bar{u}_m^{\mu_m}\vert^2 =
\vert u_1^\delta \vert^2 + \cdots + \vert u_m^\delta \vert^2\leq \rho^2,
\end{equation*}
and since $\{f_t\}$ is uniformly locally tame,
\begin{align*}
\frac{\partial (f_t)_\Delta}{\partial z_i}(\varphi(\mathbf{u}_1^m,\mathbf{a}_{m+1}^n,\bar{\mathbf{u}}_1^m,\bar{\mathbf{a}}_{m+1}^n))\not=0.
\end{align*}
Therefore, as above, the relation (\ref{getcontradiction}) implies $\nu_i=\mu_i$, which is a contradiction.
\end{proof}

Altogether, the family $\{\tilde f_t\}$ is admissible.

\bibliographystyle{amsplain}

\end{document}